\documentclass[preprint,nonatbib]{elsarticle}
\usepackage{fullpage} 

\makeatletter
\let\c@author\relax
\makeatother

\usepackage[table]{xcolor}
\usepackage{amsmath}
\usepackage{amsthm}
\usepackage{amssymb} 
\usepackage{hyperref}
\usepackage{graphicx}
\usepackage{float}
\usepackage{tikz-cd}
\usepackage{caption}
\usepackage{subcaption}
\usepackage{array}
\usepackage{extarrows}
\usepackage{multirow}
\usepackage{tikz}
\usepackage{makecell}
\usepackage{hhline}
\usepackage{cleveref}
\usetikzlibrary{cd}

\DeclareMathOperator{\LS}{2-RP}
\DeclareMathOperator{\LC}{3-RP}
\DeclareMathOperator{\RP}{-RP}

\allowdisplaybreaks[3]

\newtheorem{theorem}{Theorem}[section]
\newtheorem{lemma}[theorem]{Lemma}

\newtheorem{corollary}[theorem]{Corollary}

\theoremstyle{definition}
\newtheorem{definition}[theorem]{Definition}
\newtheorem{example}[theorem]{Ex.}
\newtheorem{construction}[theorem]{Construction}
\newtheorem{append}[theorem]{}

\numberwithin{equation}{section}

\usepackage{biblatex} %Imports biblatex package
\hypersetup{pdfauthor=author}

\makeatletter
\def\ps@pprintTitle{%
  \let\@oddhead\@empty
  \let\@evenhead\@empty
  \def\@oddfoot{\reset@font\hfil\thepage\hfil}
  \let\@evenfoot\@oddfoot
}
\makeatother

\begin{document}

\title{Latin hypercubes realizing integer partitions}

\author[1]{Diane Donovan}
\ead{dmd@maths.uq.edu.au}
\author[1]{Tara Kemp\corref{cor1}}
\ead{t.kemp@uq.net.au}
\author[1]{James Lefevre}
\ead{j.lefevre@uq.edu.au}

\affiliation[1]{organization={School of Mathematics and Physics, ARC Centre of Excellence, Plant Success in Nature and Agriculture},
    addressline={The University of Queensland},
    city={Brisbane},
    postcode={4072},
    country={Australia}}

\cortext[cor1]{Corresponding author}

\begin{abstract}
    For an integer partition $h_1 + \dots + h_n = N$, a 2-realization of this partition is a latin square of order $N$ with disjoint subsquares of orders $h_1,\dots,h_n$. The existence of 2-realizations is a partially solved problem posed by Fuchs. In this paper, we extend Fuchs' problem to $m$-ary quasigroups, or, equivalently, latin hypercubes. We construct latin cubes for some partitions with at most two distinct parts and highlight how the new problem is related to the original.
\end{abstract}


\maketitle

\section{Introduction}

L. Fuchs asked the following question \cite{keedwell2015latin}: If $n$ is any positive integer and $n=n_1+n_2+\dots+n_k$ any fixed partition of $n$, is it possible to find a quasigroup $(Q,\circ)$ of order $n$ which contains subquasigroups on the sets $Q_1,Q_2,\dots,Q_k$ of orders $n_1,n_2,\dots,n_k$ respectively, whose set theoretical union is $Q$? The problem of finding such a quasigroup is equivalent to finding a latin square of order $n$ with disjoint subsquares of orders $n_1,n_2,\dots,n_k$ (these subsquares must be row, column and symbol disjoint).

It is easy to find partitions for which such a quasigroup does not exist, however the problem of existence is far from complete. Much work has been done by considering latin squares, which are equivalent structures to quasigroups, and such a latin square is known as a realization or partitioned incomplete latin square (PILS).

For partitions with at most four parts, existence has been completely determined. Many cases have been studied for partitions with 5 or 6 parts, however the problem is not yet solved. For any partition, existence has been determined when the parts are of at most two distinct sizes.

The same problem can be posed for $m$-ary quasigroups, which are equivalent to $m$-dimensional latin hypercubes. In this paper, we investigate this problem by first considering latin cubes and also by examining the relationships between latin hypercubes of different dimenions.

\begin{definition}
    A \emph{latin square} of order $N$ is an $N\times N$ array, $L$, on a set of $N$ symbols such that each symbol occurs exactly once in each row and column. An entry of $L$ is defined by its row and column position as $L(i,j)$.

    An $m\times m$ sub-array of $L$ that is itself a latin square of order $m$ is a \emph{subsquare} of $L$. Subsquares are disjoint if they share no row, column or symbol of $L$.
\end{definition}

For any positive integer $a$, we define $[a]$ to be the set of integers $\{1,2,\dots,a\}$. Given another integer $b$, $b+[a]$ is the set $\{b+1,b+2,\dots,b+a\}$.

We will assume that all latin squares of order $N$ use the symbols of $[N]$.

\begin{example}
\label{order5square}
The array in \autoref{fig:2,1,1,1} is a latin square of order 5, and the highlighted cells indicate a subquare of order 2 and three subsquares of order 1.
\begin{figure}[h]
    \centering
    $\arraycolsep=4pt\begin{array}{|c|c|c|c|c|} \hhline{|*{5}{-|}}
    \cellcolor{lightgray}1 & \cellcolor{lightgray}2 & 4 & 5 & 3\\ \hhline{|*{5}{-|}}
    \cellcolor{lightgray}2 & \cellcolor{lightgray}1 & 5 & 3 & 4\\ \hhline{|*{5}{-|}}
    4 & 5 & \cellcolor{lightgray}3 & 2 & 1\\ \hhline{|*{5}{-|}}
    5 & 3 & 1 & \cellcolor{lightgray}4 & 2\\ \hhline{|*{5}{-|}}
    3 & 4 & 2 & 1 & \cellcolor{lightgray}5\\ \hhline{|*{5}{-|}}\end{array}$
    \caption{}
    \label{fig:2,1,1,1}
\end{figure}
\end{example}

\begin{definition}
    An \emph{integer partition} of $N$ is a sequence of non-increasing integers $P = (h_1,h_2,h_3,\dots)$ where $\sum_{i=1}^\infty h_i = N$.
\end{definition}

An integer partition $P = (h_1,h_2,\dots,h_n)$ will be represented by $(h_1h_2\dots h_n)$ or $(h_{i_1}^{\alpha_1}h_{i_2}^{\alpha_2}\dots h_{i_m}^{\alpha_m})$, where there are $\alpha_j$ copies of $h_{i_j}$ in the partition.

\begin{definition}
     For an integer partition $P = (h_1,h_2,\dots,h_n)$ of $N$ with $h_1\geq h_2\geq \dots\geq h_n > 0$, a \emph{2-realization} of $P$, denoted $\LS(h_1h_2\dots h_n)$, is a latin square of order $N$ with pairwise disjoint subsquares of order $h_i$ for each $i\in [n]$.
     We say that a 2-realization is in \emph{normal form} if the subsquares are along the main diagonal and the $i^{th}$ subsquare, $H_i$, is of order $h_i$.
\end{definition}

The latin square in Example \ref{order5square} is a 2-realization of the partition $(2^11^3)$.

The existence of 2-realizations for $n=1,2,3,4$ has been completely determined in \cite{heinrich2006latin}.
\begin{theorem}
\label{small n squares}
    Take a partition $(h_1,h_2,\dots,h_n)$ of $N$ with $h_1\geq h_2\geq \dots\geq h_n > 0$. Then a $\LS(h_1h_2\dots h_n)$
    \begin{itemize}
        \item always exists when $n=1$;
        \item never exists when $n=2$;
        \item exists when $n=3$ if and only if $h_1=h_2=h_3$;
        \item exists when $n=4$ if and only if $h_1=h_2=h_3$, or $h_2=h_3=h_4$ with $h_1\leq 2h_4$.
    \end{itemize}
\end{theorem}

The problem has also been solved for partitions with subsquares of at most two distinct orders. D{\'e}nes and P{\'a}sztor \cite{denes1963some} gave the result for subsquares all of the same order and the result for two different orders was considered by Heinrich \cite{heinrich1982disjoint} and finished by Kuhl and Schroeder \cite{kuhl2018latin}.
\begin{theorem}
\label{a^n square}
    For $n\geq 1$ and $a\geq 1$, a $\LS(a^n)$ exists if and only if $n\neq 2$.
\end{theorem}

The following theorem completes the problem for all partitions with at most two distinct orders, considering only the cases which haven't already been covered in the previous theorems.
\begin{theorem}
\label{squaresatmost2}
    For $a>b>0$ and $n>4$, a $\LS(a^ub^{n-u})$ exists if and only if $u\geq 3$, or $0< u < 3$ and $a\leq (n-2)b$.
\end{theorem}

Colbourn \cite{colbourn2018latin} has determined two necessary, but not sufficient, conditions for the existence of a 2-realization.
\begin{theorem}
\label{squarecondition1}
    If a $\LS(h_1h_2\dots h_n)$ exists, then $h_1\leq\sum_{i=3}^nh_i$.
\end{theorem}

\begin{theorem}
\label{squarecondition2}
    If a $\LS(h_1h_2\dots h_n)$ exists, then $$N^2 - \sum_{i=1}^nh_i^2\geq 3\Bigg( \sum_{i\in D}h_i\Bigg)\Bigg(\sum_{j\in\overline{D}}h_j\Bigg)$$ for any $D\subseteq\{1,2,\dots,n\}$ and $\overline{D} = \{1,2,\dots,n\}\setminus D$.
\end{theorem}

Fuchs' problem considers subgroups of quasigroups with binary operations which are equivalent to latin squares, and so this problem can be extended to quasigroups with ternary operations by looking at latin cubes. These are also known as permutation cubes.

\begin{definition}
    A \emph{latin cube} $C$ of order $N$ is an $N\times N\times N$ array on a symbol set of size $N$, with entries defined by $C(j_1,j_2,j_3)$, such that if $(j_1,j_2,j_3)$ and $(j'_1,j'_2,j'_3)$ differ in exactly one position, then $C(j_1,j_2,j_3) \neq C(j'_1,j'_2,j'_3)$.

    If $C$ is instead an $r\times s\times t$ array with $r,s,t\leq N$, then we call $C$ a \emph{latin box} of order $r$ by $s$ by $t$.

    We will consider a \emph{line} of $C$ to be a set of $N$ cells with two fixed coordinates. A \emph{column}, \emph{row} or \emph{file} is thus a line represented by $([N],j_2,j_3)$, $(j_1,[N],j_3)$ or $(j_1,j_2,[N])$ respectively. A \emph{layer} of $C$ is the set of cells with only one fixed coordinate. Layers can thus be represented as $(j_1,[N],[N])$, $([N],j_2,[N])$ or $([N],[N],j_3)$.
    
    Similar to latin squares, a \emph{subcube} of $C$ is an $m\times m\times m$ sub-array which is itself a latin cube of order $m$. Subcubes are disjoint if they share no row, column, file or symbol.

    As with latin squares, we will assume that all latin cubes use the symbols of $[N]$.
\end{definition}

\begin{definition}
    For an integer partition $P = (h_1,h_2,\dots,h_n)$ of $N$ with $h_1\geq h_2\geq \dots\geq h_n > 0$, a \emph{3-realization} of $P$, denoted $\LC(h_1h_2\dots h_n)$, is a latin cube of order $N$ with pairwise disjoint subcubes of order $h_i$ for each $i\in [n]$.

    As with 2-realizations, a 3-realization is in \emph{normal form} if the subcubes are along the main diagonal (cells of the form $(i,i,i)$ for all $i\in [N]$ and the additional cells needed to form the subsquares) and the $i^{th}$ subcube, $H_i$, is of order $h_i$.
\end{definition}

\begin{example}
\label{2^21^1}
The arrays in \autoref{fig:2,2,1 cube} are the layers of a $\LC(2^21^1)$. The subcubes are highlighted.
    \begin{figure}[h]
    \centering
    $\arraycolsep=4pt\begin{array}{|c|c|c|c|c|} \hline
    \cellcolor{lightgray}1 & \cellcolor{lightgray}2 & 3 & 5 & 4\\ \hline
    \cellcolor{lightgray}2 & \cellcolor{lightgray}1 & 4 & 3 & 5\\ \hline
    3 & 4 & 5 & 1 & 2\\ \hline
    4 & 5 & 1 & 2 & 3\\ \hline
    5 & 3 & 2 & 4 & 1\\ \hline \end{array}$
    \quad
    $\arraycolsep=4pt\begin{array}{|c|c|c|c|c|} \hline
    \cellcolor{lightgray}2 & \cellcolor{lightgray}1 & 4 & 3 & 5\\ \hline
    \cellcolor{lightgray}1 & \cellcolor{lightgray}2 & 5 & 4 & 3\\ \hline
    5 & 3 & 1 & 2 & 4\\ \hline
    3 & 4 & 2 & 5 & 1\\ \hline
    4 & 5 & 3 & 1 & 2\\ \hline \end{array}$
    \quad
    $\arraycolsep=4pt\begin{array}{|c|c|c|c|c|} \hline
    4 & 3 & 5 & 2 & 1\\ \hline
    3 & 5 & 2 & 1 & 4\\ \hline
    1 & 2 & \cellcolor{lightgray}3 & \cellcolor{lightgray}4 & 5\\ \hline
    5 & 1 & \cellcolor{lightgray}4 & \cellcolor{lightgray}3 & 2\\ \hline
    2 & 4 & 1 & 5 & 3\\ \hline \end{array}$
    \quad
    $\arraycolsep=4pt\begin{array}{|c|c|c|c|c|} \hline
    5 & 4 & 2 & 1 & 3\\ \hline
    4 & 3 & 1 & 5 & 2\\ \hline
    2 & 5 & \cellcolor{lightgray}4 & \cellcolor{lightgray}3 & 1\\ \hline
    1 & 2 & \cellcolor{lightgray}3 & \cellcolor{lightgray}4 & 5\\ \hline
    3 & 1 & 5 & 2 & 4\\ \hline \end{array}$
    \quad
    $\arraycolsep=4pt\begin{array}{|c|c|c|c|c|} \hline
    3 & 5 & 1 & 4 & 2\\ \hline
    5 & 4 & 3 & 2 & 1\\ \hline
    4 & 1 & 2 & 5 & 3\\ \hline
    2 & 3 & 5 & 1 & 4\\ \hline
    1 & 2 & 4 & 3 & \cellcolor{lightgray}5\\ \hline \end{array}$
    \caption{}
    \label{fig:2,2,1 cube}
    \end{figure}
    
\end{example}

The problem of existence of a $\LC(h_1h_2\dots h_n)$ is different to that of a $\LS(h_1h_2\dots h_n)$. To see this, consider the latin cube in \autoref{fig:1,1 cube}. This latin cube is a 3-realization of the partition $P = (1,1)$, so a $\LC(1^2)$. From \Cref{small n squares}, we know that a $\LS(1^2)$ does not exist. Thus, there are partitions which can be realized in latin cubes but not in latin squares.

\begin{figure}[h]
        \centering

        \begin{tikzpicture}
            \filldraw[color = lightgray, fill=lightgray] (0,0) rectangle (0.75,-0.75);
            \draw (0,0) -- (1.5,0);
            \draw (0,0) -- (0,-1.5);
            \draw (1.5,0) -- (1.5,-1.5);
            \draw (0,-1.5) -- (1.5,-1.5);
            \draw (0,-0.75) -- (1.5,-0.75);
            \draw (0.75,0) -- (0.75,-1.5);
            \node at (0.375,-0.375) {1};
            \node at (1.125,-0.375) {2};
            \node at (0.375,-1.125) {2};
            \node at (1.125,-1.125) {1};
        \end{tikzpicture}
        \quad
        \begin{tikzpicture}
            \filldraw[color = lightgray, fill=lightgray] (0.75,-0.75) rectangle (1.5,-1.5);
            \draw (0,0) -- (1.5,0);
            \draw (0,0) -- (0,-1.5);
            \draw (1.5,0) -- (1.5,-1.5);
            \draw (0,-1.5) -- (1.5,-1.5);
            \draw (0,-0.75) -- (1.5,-0.75);
            \draw (0.75,0) -- (0.75,-1.5);
            \node at (0.375,-0.375) {2};
            \node at (1.125,-0.375) {1};
            \node at (0.375,-1.125) {1};
            \node at (1.125,-1.125) {2};
        \end{tikzpicture}
        \caption{}
        \label{fig:1,1 cube}
    \end{figure}

\section{Constructions}

We now give some construction methods for latin cubes which will be used later to find 3-realizations. Construction \ref{construction} is quite general, but this allows it to be used to construct more specific latin cubes in later sections. Construction \ref{inflation} is a generalisation of a construction method used for latin squares.

\begin{construction}
\label{construction}
    Given any order $N$ latin square $L$, we can construct an order $N$ latin cube $C$ as follows.
    
    Let the entries of $L$ be determined by $L(r,c)$. Construct $C$ as $C(r,c,l) = L(L(r,l),c)$. Observe that if $C(r,c,l) = C(r',c,l)$ then $L(L(r,l),c) = L(L(r',l),c)$, and since $L$ is a latin square, $L(r,l) = L(r',l)$. Again, this implies that $r=r'$. If $C(r,c,l) = C(r,c',l)$, then $L(L(r,l),c) = L(L(r,l),c')$, and since $L$ is a latin square, $c = c'$. Also, if $C(r,c,l) = C(r,c,l')$, then $L(L(r,l),c) = L(L(r,l'),c)$. As before, this implies that $L(r,l) = L(r,l')$ and so $l = l'$.
    
    Thus, $C$ is a latin cube.
\end{construction}

\begin{construction}
\label{inflation}
    Given any latin cube, $C$, of order $N$, we can construct a latin cube of order $tN$, for any $t\in\mathbb{Z}^+$.
    
    To do this, take a latin cube $T$ of order $t$ on the symbol set $\{1,2,\dots,t\}$ and an empty $tN\times tN$ array $C'$. For every entry $C(i,j,k)$ of $C$, place a copy of $T$ with the symbols $(C(i,j,k)-1)t + [t]$ across the cells $C'(i',j',k')$ of $C'$ where $i'\in (i-1)t + [t]$, $j'\in (j-1)t + [t]$ and $k'\in (k-1)t + [t]$.

    In this way, every entry of $C$ corresponds to a copy of $T$ in $C'$ with a set of symbols determined by the symbol in $C$.
    
    This is the analogue of a construction method for latin squares known as \emph{inflation}.
\end{construction}

This method of construction gives some immediate results for 3-realizations. Similarly to what can be done for latin squares, if a latin cube has a transversal of $N$ symbols on the main diagonal, then an inflation of this by $a$ will be a $\LC(a^N)$, since the subsquares correspond to copies of $T$ with distinct symbols. Also, if a $\LC(h_1\dots h_n)$ exists, then an inflation of this is a $\LC((ah_1)^1\dots (ah_n)^1)$.

\emph{Prolongation} is another method which has been used to find $2-$realizations. The article \cite{heinrich1982prolongation} defines a generalisation of this method for latin cubes but it is unknown whether it can be used for latin cubes of any order.

\subsection{Outline Rectangles and Boxes}

\begin{definition}
    Given partitions $P,Q,R$ of $N$, where $P = (p_1,\dots,p_u)$, $Q = (q_1,\dots,q_v)$, $R = (r_1,\dots,r_t)$, let $O$ be a $u\times v$ array of multisets, with elements from $[t]$. For $i\in [u]$ and $j\in[v]$, let $O(i,j)$ be the multiset of symbols in cell $(i,j)$ and let $|O(i,j)|$ be the number of symbols in the cell, including repetition.

    Then $O$ is an \emph{outline rectangle} associated to $(P,Q,R)$ if
    \begin{enumerate}
        \item $|O(i,j)| = p_iq_j$, for all $(i,j)\in[u]\times[v]$;
        \item symbol $l\in[t]$ occurs $p_ir_l$ times in the row $(i,[v])$;
        \item symbol $l\in[t]$ occurs $q_jr_l$ times in the column $([u],j)$.
    \end{enumerate}
\end{definition}

\begin{example}
For partitions $P=(4^11^2)$, $Q=(3^12^11^1)$ and $R=(2^3)$, the array in \autoref{fig:outline square} is an outline rectangle associated to $(P,Q,R)$.
    \begin{figure}[h]
        \centering
        $\begin{array}{|ccc|cc|c|} \hline
        1 & 1 & 1 & 2 & 2 & 1 \\
        1 & 1 & 1 & 2 & 2 & 1 \\
        2 & 2 & 3 & 3 & 3 & 2 \\
        3 & 3 & 3 & 3 & 3 & 2 \\ \hline
        2 & 2 & 3 & 1 & 1 & 3 \\ \hline
        2 & 2 & 3 & 1 & 1 & 3 \\ \hline
        \end{array}$
        \caption{}
        \label{fig:outline square}
    \end{figure}
\end{example}

\begin{definition}
    Given partitions $P,Q,R$ of $N$, where $P = (p_1,\dots,p_u)$, $Q = (q_1,\dots,q_v)$ and $R = (r_1,\dots,r_w)$, and a latin square $L$ of order $N$, the \emph{reduction modulo $(P,Q,R)$} of $L$, denoted $O$, is the $u\times v$ array of multisets
    obtained by amalgamating rows $(p_1 + \dots + p_{i-1}) + [p_i]$ for all $i\in[u]$, columns $(q_1 + \dots + q_{j-1}) + [q_j]$ for all $j\in[v]$, and symbols $(r_1 + \dots + r_{k-1}) + [r_k]$ for all $k\in[w]$.

    When amalgamating symbols, for $k\in[w]$ we will map all symbols in $(r_1 + \dots + r_{k-1}) + [r_k]$ to symbol $k$.
\end{definition}

\begin{example}
The second array in \autoref{fig:reduction} is a reduction modulo $(P,Q,R)$ of the latin square in Example \ref{order5square} (the left array) with $P=(2^21^1)$, $Q=(4^11^1)$ and $R=(2^11^3)$. In this reduction of $L$, symbols 1 and 2 have been amalgamated into symbol 1, and for the remaining symbols, symbol $i$ became $i-1$.
    \begin{figure}[h]
        \centering
        $\arraycolsep=4pt\begin{array}{|c|c|c|c|c|} \hhline{|*{5}{-|}}
    \cellcolor{lightgray}1 & \cellcolor{lightgray}2 & 4 & 5 & 3\\ \hhline{|*{5}{-|}}
    \cellcolor{lightgray}2 & \cellcolor{lightgray}1 & 5 & 3 & 4\\ \hhline{|*{5}{-|}}
    4 & 5 & \cellcolor{lightgray}3 & 2 & 1\\ \hhline{|*{5}{-|}}
    5 & 3 & 1 & \cellcolor{lightgray}4 & 2\\ \hhline{|*{5}{-|}}
    3 & 4 & 2 & 1 & \cellcolor{lightgray}5\\ \hhline{|*{5}{-|}}\end{array}$\quad \quad
        $\arraycolsep=4pt\begin{array}{|cccc|c|} \hline
        1 & 1 & 1 & 1 & 2 \\
        2 & 3 & 4 & 4 & 3 \\ \hline
        1 & 1 & 2 & 2 & 1 \\
        3 & 3 & 4 & 4 & 1 \\ \hline
        1 & 1 & 2 & 3 & 4 \\ \hline
        \end{array}$
        \caption{}
        \label{fig:reduction}
    \end{figure}
\end{example}

The following theorem, and its proof, can be found in Hilton \cite{hilton1980reconstruction}.

\begin{theorem}
\label{outline rectangle to square}
    For every outline rectangle $O$, there is a latin square $L$ of order $N$ and partitions $P$, $Q$ and $R$ such that $O$ is the reduction modulo $(P, Q, R)$ of $L$.
\end{theorem}

This result provides a construction method which has been used to prove the existence of many 2-realizations. To construct a latin square $L$, we do not need to find $L$ itself. It is sufficient, by \Cref{outline rectangle to square}, to find an outline rectangle $O$ which is a reduction of $L$. For a realization of partition $P = (h_1,\dots,h_n)$, a reduction modulo $(P,P,P)$ with each cell $(i,i)$ on the main diagonal filled with $h_i^2$ copies of symbol $i$ will force the resulting latin square $L$ to have the required subsquares.

Given the usefulness of outline rectangles in finding 2-realizations, we will define the analogous concept for latin cubes.

\begin{definition}
    Given partitions $P,Q,R,S$ of $N$, where $P = (p_1,\dots,p_u)$, $Q = (q_1,\dots,q_v)$, $R = (r_1,\dots,r_w)$, $S = (s_1,\dots,s_t)$, let $O$ be a $u\times v\times w$ array of multisets, with elements from $[t]$. For $i\in [u]$, $j\in[v]$ and $k\in[w]$, let $O(i,j,k)$ be the multiset of symbols in cell $(i,j,k)$ and let $|O(i,j,k)|$ be the number of symbols in the cell, including repetition.

    Then $O$ is an \emph{outline box} associated to $(P,Q,R,S)$ if
    \begin{enumerate}
        \item symbol $l\in[t]$ occurs $N^2s_l$ times in $O$;
        \item $|O(i,j,k)| = p_iq_jr_k$, for all $(i,j,k)\in[u]\times[v]\times[w]$;
        \item symbol $l\in[t]$ occurs $p_iq_js_l$ times in the line $(i,j,[w])$;
        \item symbol $l\in[t]$ occurs $p_ir_ks_l$ times in the line $(i,[v],k)$;
        \item symbol $l\in[t]$ occurs $q_jr_ks_l$ times in the line $([u],j,k)$.
    \end{enumerate}
\end{definition}

\begin{definition}
    For partitions $P,Q,R,S$ of $N$, where $P = (p_1,\dots,p_u)$, $Q = (q_1,\dots,q_v)$ $R = (r_1,\dots,r_w)$ and $S = (s_1,\dots,s_t)$, and a latin cube $C$ of order $N$, the \emph{reduction modulo $(P,Q,R,S)$} of $C$, denoted $O$, is the $u\times v\times w$ array of multisets
    obtained by amalgamating layers $(j_1,[N],[N])$, $([N],j_2,[N])$ and $([N],[N],j_3)$ for all $j_1\in(p_1 + \dots + p_{i-1}) + [p_i]$ for all $i\in[u]$, $j_2\in(q_1 + \dots + q_{j-1}) + [q_j]$ for all $j\in[v]$ and $j_3\in(r_1 + \dots + r_{k-1}) + [r_k]$ for all $k\in[w]$, and symbols $(s_1 + \dots + s_{l-1}) + [s_l]$ for all $l\in[t]$.
\end{definition}

Unlike with outline rectangles, these outline boxes are not always the reduction of a latin cube. Kochol \cite{kochol1989latin} has shown that not all $k\times n\times n$ latin boxes can be extended to a latin cube and we will use this result in the following proof.

\begin{theorem}
    Not every outline box $O$ associated to $(P,Q,R,S)$ is the reduction of a latin cube $C$.
\end{theorem}
\begin{proof}
    Suppose that $O$ is the reduction of a latin cube $C$ modulo $(P,Q,R,S)$.

    Consider the outline box where $Q = R = S = (1^N)$ and $P = (N-k,1^k)$. Then $O$ is obtained by amalgamating $N-k$ layers of $C$. Thus, $O$ has $k$ layers which form a $k\times N\times N$ latin box. Since $C$ only has those layers amalgamated, $C$ can be obtained by completing this latin box to a full latin cube. However, from above, this is not always possible.
    
    Thus, there are outline boxes which are not the reduction of a latin cube.
\end{proof}

\begin{example}
    The outline box associated to $(P,Q,R,S)$ given in \autoref{fig:outline box not work}, with $P = R = S = (1^4)$ and $Q = (2^2)$ and each array representing a layer, is not the reduction of a latin cube.
    \begin{figure}[h]
        \centering
    $\arraycolsep=4pt\begin{array}{|cc|cc|} \hline
    3 & 4 & 1 & 2\\ \hline
    3 & 4 & 1 & 2\\ \hline
    1 & 2 & 3 & 4\\ \hline
    1 & 2 & 3 & 4\\ \hline \end{array}$
    \quad
    $\arraycolsep=4pt\begin{array}{|cc|cc|} \hline
    2 & 4 & 1 & 3\\ \hline
    2 & 3 & 1 & 4\\ \hline
    1 & 3 & 2 & 4\\ \hline
    1 & 4 & 2 & 3\\ \hline \end{array}$
    \quad
    $\arraycolsep=4pt\begin{array}{|cc|cc|} \hline
    1 & 3 & 2 & 4\\ \hline
    1 & 2 & 3 & 4\\ \hline
    2 & 4 & 1 & 3\\ \hline
    3 & 4 & 1 & 2\\ \hline \end{array}$
    \quad
    $\arraycolsep=4pt\begin{array}{|cc|cc|} \hline
    1 & 2 & 3 & 4\\ \hline
    1 & 4 & 2 & 3\\ \hline
    3 & 4 & 1 & 2\\ \hline
    2 & 3 & 1 & 4\\ \hline \end{array}$
    \caption{}
    \label{fig:outline box not work}
    \end{figure}
\end{example}

\section{Some conditions on 3-realizations}

In this section, we give some conditions on the existence of a $\LC(h_1h_2\dots h_n)$. We begin with a theorem that shows an important relationship between 2-realizations and 3-realizations.

\begin{theorem}
\label{CubefromSquare}
    If there exists a $\LS(h_1\dots h_n)$, then there exists a $\LC(h_1\dots h_n)$.
\end{theorem}
\begin{proof}
    Suppose that $L$ is a $\LS(h_1\dots h_n)$. First, we assume that $L$ is in normal form and that the entries of subsquare $H_i$, corresponding to $h_i$, are in $S_i = (h_1+\dots +h_{i-1}) + [h_i]$. Thus, for all $(r,c)\in S_i\times S_i$, we have $L(r,c)\in S_i$. If this is not the case, then the rows, columns and symbols can be permuted to satisfy this condition.

    We construct a latin cube $C$ using the method given in Construction \ref{construction}.

    We now check that the subcubes are maintained in $C$. For any part $h_i$, if $(r,c,l)\in S_i\times S_i\times S_i$ then $L(r,l) = a\in S_i$ and $L(a,c)\in S_i$ by our earlier assumption. Thus, $C(r,c,l) = L(L(r,l),c)\in S_i$. Therefore, the cubes formed by the columns, rows and layers in $S_i$ are filled with only the elements of $S_i$. So $C$ has the required subcubes.
\end{proof}

This theorem means that we can use all results on the existence of $\LS(h_1h_2\dots h_n)$ when solving the problem in three dimensions.

We now provide analogues to the two necessary conditions given for $\LS(h_1h_2\dots h_n)$.

The first condition is similar to that of \Cref{squarecondition1}, by placing a bound on the size of the largest subcube.

\begin{theorem}
\label{condition1}
    If a $\LC(h_1\dots h_n)$ exists and $n\geq 2$, then $h_i\leq \frac{N}{2}$ for all $1\leq i\leq n$.
\end{theorem}
\begin{proof}
    Suppose that there is a latin cube $\LC(h_1\dots h_n)$. Rearrange the layers of the cube so that the subcube of order $h_i$ is in the cells $(i,j,k)$ with $i,j,k\in [h_i]$, making it the first subcube on the main diagonal. Let this subcube be $H_i$. Then the cells $(i,j,1)$ form the latin square shown in \Cref{table}, where the subarray $A_1$ is a $h_i\times h_i$ array and $A_4$ is an $(N-h_i)\times (N-h_i)$ array.
    \begin{figure}[h]
        \centering

        \begin{tikzpicture}
            \draw (0,0) -- (1.5,0);
            \draw (0,0) -- (0,-1.5);
            \draw (1.5,0) -- (1.5,-1.5);
            \draw (0,-1.5) -- (1.5,-1.5);
            \draw (0,-0.75) -- (1.5,-0.75);
            \draw (0.75,0) -- (0.75,-1.5);
            \node at (0.375,-0.375) {$A_1$};
            \node at (1.125,-0.375) {$A_2$};
            \node at (0.375,-1.125) {$A_3$};
            \node at (1.125,-1.125) {$A_4$};
        \end{tikzpicture}

        \caption{}
        \label{table}
    \end{figure}
    
    Observe that the cells of $A_1$ are within the subcube $H_i$ of order $h_i$. Thus, the symbols of $H_i$ cannot appear in $A_3$. These symbols need to appear in every row and column, so each row of the subarray $A_4$ must contain all $h_i$ symbols in $H_i$. This gives that $h_i \leq N-h_i$.
\end{proof}

The next condition is similar to \Cref{squarecondition2}.

\begin{theorem}
\label{condition2}
    If a $\LC(h_1\dots h_n)$ exists then $N^3 - \sum_{i=1}^n h_i^3 \geq N^2\sum_{i\in D} h_i + 3(\sum_{i\in D} h_i^2)(N-\sum_{i\in D} h_i) - \sum_{i\in D} h_i^3$ for any subset $D$ of $\{1,\dots,n\}$.
\end{theorem}
\begin{proof}
    Suppose that a $\LC(h_1\dots h_n)$, $C$, exists. We assume that $C$ is in normal form and that subcube $H_i$ contains the symbols of $S_i = (h_1+\dots+h_{i-1}) + [h_i]$. For some fixed subset $D$ of $[n]$, let $S = \cup_{i\in D}S_i$.

    Fix $i\in D$ and let $L_i$ be the set of lines defined by $(s,s',[N])$, $(s,[N],s')$ and $([N],s,s')$ for all $s,s'\in S_i$. The only cells appearing in more than one of these lines are the cells of $H_i$, which appear three times. Also, every symbol of $S$ appears within every line, with the symbols of $S_i$ being only within $H_i$.

    Consider the union of these sets $L_i$ for all $i\in D$. Clearly the cells in the $L_i$ are disjoint. Thus, the number of times that the symbols of $S$ occur in the union is

    $$\sum_{i\in D}h_i^3 + 3\sum_{\substack{i,j\in D \\ i\neq j}}h_i^2h_j.$$

    Therefore, the number of times they occur not within these lines is

    $$N^2|S| - \sum_{i\in D}h_i^3 - 3\sum_{\substack{i,j\in D \\ i\neq j}}h_i^2h_j = N^2\sum_{i\in D}h_i - \sum_{i\in D}h_i^3 - 3\sum_{\substack{i,j\in D \\ i\neq j}}h_i^2h_j = N^2\sum_{i\in D}h_i - (\sum_{i\in D}h_i)^3 + 6\sum_{\substack{i,j,k\in D \\ i<j<k}}h_ih_jh_k.$$

    We now count the number of cells not within these lines which can contain the symbols of $S$. This is all of the cells not within the lines, and not within the subcubes $H_j$ for $j\in \overline{D}$. The number of cells is

    $$N^3 - \sum_{i=1}^n h_i^3 - \sum_{i\in D}3h_i^2(N-h_i) = N^3 - \sum_{i=1}^n h_i^3 - 3N\sum_{i\in D}h_i^2 + 3(\sum_{i\in D}h_i^2)(\sum_{i\in D}h_i) + \sum_{i\in D}h_i^3 - (\sum_{i\in D}h_i)^3 + 6\sum_{\substack{i,j,k\in D \\ i<j<k}}h_ih_jh_k.$$

    Therefore, combining these two numbers,
    
    $$N^3 - \sum_{i=1}^n h_i^3 \geq N^2\sum_{i\in D} h_i + 3(\sum_{i\in D} h_i^2)(N-\sum_{i\in D} h_i) - \sum_{i\in D} h_i^3.$$

\end{proof}

We now begin to update the established existence results of 2-realizations for 3-realizations.

\begin{theorem}
\label{a^n}
    A $\LC(a^n)$ exists for all $a,n\in\mathbb{Z}^+$.
\end{theorem}
\begin{proof}
    $\LC(a^n)$ exists for all $n\neq 2$ by \Cref{a^n square} and \Cref{CubefromSquare}.

    For $\LC(a^2)$, take the latin cube $C$ of order 2 given in \autoref{fig:1,1 cube}.

    Now take a latin cube $A$ of order $a$ and inflate $C$ with $A$ to get a latin cube of order $2a$ with 2 disjoint subcubes of order $a$ on the main diagonal.
\end{proof}

The following theorem provides both a condition for existence and also a method for constructing 3-realizations of partitions.

\begin{theorem}
\label{halving condition}
    A $\LC(h_1\dots h_n)$ with $h_1 = \frac{N}{2}$ exists if and only if a $\LC(h_2\dots h_n)$ exists.
\end{theorem}
\begin{proof}
    If $C_1$ is a $\LC(h_2\dots h_n)$, then let $h_1 = \sum_{i=2}^n h_i$. We construct $C_2$ by taking a $\LC(h_1^2)$, which we know exists from \Cref{a^n}, and replacing one of the subcubes of order $h_1$ with $C_1$. Thus, $C_2$ is a $\LC(h_1\dots h_n)$ with $h_1 = \frac{N}{2}$.

    If $C_1$ is a $\LC(h_1\dots h_n)$ with $h_1 = \frac{N}{2} = \sum_{i=2}^n h_i$, then it can be represented as in \Cref{ABCDEFGH}, where all cells are $h_1\times h_1\times h_1$ arrays, $A$ is the order $h_1$ subcube $H_1$ and $H$ contains all other subcubes.

    \begin{figure}[h]
        \centering

        \begin{tikzpicture}
            \draw (0,0) -- (1.5,0);
            \draw (0,0) -- (0,-1.5);
            \draw (1.5,0) -- (1.5,-1.5);
            \draw (0,-1.5) -- (1.5,-1.5);
            \draw (0,-0.75) -- (1.5,-0.75);
            \draw (0.75,0) -- (0.75,-1.5);
            \node at (0.375,-0.375) {A};
            \node at (1.125,-0.375) {B};
            \node at (0.375,-1.125) {C};
            \node at (1.125,-1.125) {D};
        \end{tikzpicture}
        \quad
        \begin{tikzpicture}
            \draw (0,0) -- (1.5,0);
            \draw (0,0) -- (0,-1.5);
            \draw (1.5,0) -- (1.5,-1.5);
            \draw (0,-1.5) -- (1.5,-1.5);
            \draw (0,-0.75) -- (1.5,-0.75);
            \draw (0.75,0) -- (0.75,-1.5);
            \node at (0.375,-0.375) {E};
            \node at (1.125,-0.375) {F};
            \node at (0.375,-1.125) {G};
            \node at (1.125,-1.125) {H};
        \end{tikzpicture}
        \caption{}
        \label{ABCDEFGH}
    \end{figure}
    
    Since $C_1$ is a latin cube, all symbols from $H_1$ must appear in the subarrays $A$, $D$, $F$ and $G$, with the remaining cells only containing symbols from the remaining subcubes. Thus, $H$ is a latin cube on the symbols from subcubes $H_2,\dots,H_n$. Therefore $H$ is a $\LC(h_2\dots h_n)$.
\end{proof}

\begin{example}
    The partition $(3^21)$ satisfies the condition from \Cref{condition1} but not the condition from \Cref{condition2}.

    The partition $(10^12^11^7)$ satisfies the condition from \Cref{condition2} but not the condition from \Cref{condition1}. This shows that both of these conditions are needed as necessary conditions for the existence of a 3-realization.

    The partition $(19^110^12^11^7)$ satisfies both conditions, but the existence of a $\LC(19^110^12^11^7)$ would imply the existence of a $\LC(10^12^11^7)$ by \Cref{halving condition}. The partition $(10^12^11^7)$ cannot be realized due to \Cref{condition1}. Thus, we can conclude that the first two conditions are necessary but not sufficient to show existence.
\end{example}

\section{Partitions with at most two distinct parts}

This section will focus on partitions with two distinct parts, given by $(a^ub^{n-u})$, where $a > b$. Theorems \ref{small n squares} and \ref{squaresatmost2} give the existence without bounds of a 2-realization for many such partitions, so we need only consider the cases where $1<n<5$ or $u<3$.

\begin{lemma}
    A $\LC(h_1h_2)$ exists if and only if $h_1=h_2$.
\end{lemma}
\begin{proof}
    If $h_1=h_2$, then we are considering the partition $(h_1^2)$. By \Cref{a^n}, a $\LC(h_1^2)$ exists.
    
    Suppose that there is a latin cube $\LC(h_1h_2)$. From \Cref{condition1}, $h_1\leq \frac{N}{2}$ and $h_2\leq \frac{N}{2}$. Since $h_1+h_2 = N$, $h_1=h_2=\frac{N}{2}$.
\end{proof}

We now consider the case where $u = 1$. For latin squares, $a$ is limited by the bound $a\leq (n-1)b$ given in \Cref{squaresatmost2}, and for $n< 4$ the 2-realization does not exist. We begin with the case where $n=3$.

\begin{lemma}
\label{ab^2}
    A $\LC(a^1b^2)$ with $a>b$ exists if and only if $a\leq 2b$.
\end{lemma}
\begin{proof}
    Suppose that a $\LC(a^1b^2)$ exists, then by \Cref{condition1} we have that $a\leq 2b$. Now suppose that this is satisfied.

    If $a=2b$ then a $\LC(a^1b^2)$ can be constructed by \Cref{halving condition} with a $\LC(a^2)$.

    If $a<2b$, then we begin by constructing an outline rectangle $O$, with the structure shown in \Cref{ab^2 O}. We take the partitions $P = Q = R = (b^1(a-b)^1b^2)$ for the outline rectangle, so we are working with the symbols $\{1,2,3,4\}$. In $O$, we want the cells $A_1$, $A_2$, $A_3$ and $A_4$ to contain a total of $a^2$ copies of only symbols 1 and 2, and the cells $B'$, $C'$ and $D'$ to be filled with $b^2$ copies of 3, 4 and 1 respectively.

    \begin{figure}[h]
        \centering

        \begin{subfigure}[b]{0.3\textwidth}
            \centering
            \begin{tikzpicture}
            \draw (0,0) -- (3.5,0);
            \draw (0,0) -- (0,-3.5);
            \draw (3.5,0) -- (3.5,-3.5);
            \draw (0,-3.5) -- (3.5,-3.5);
            \draw (0,-1.5) -- (3.5,-1.5);
            \draw (1.5,0) -- (1.5,-3.5);
            \draw (2.5,0) -- (2.5,-3.5);
            \draw (0,-2.5) -- (3.5,-2.5);
            \draw (0,-1) -- (3.5,-1);
            \draw (1,0) -- (1,-3.5);
            \node at (0.5,-0.5) {A$_1$};
            \node at (1.25,-0.5) {A$_2$};
            \node at (0.5,-1.25) {A$_3$};
            \node at (1.25,-1.25) {A$_4$};
            \node at (2, -0.5) {B$'$};
            \node at (3, -0.5) {C$'$};
            \node at (2, -2) {D$'$};
            \node at (3, -3) {D$'$};
            \end{tikzpicture}
            \caption{The outline rectangle $O$}
            \label{ab^2 O}
        \end{subfigure}
        \begin{subfigure}[b]{0.3\textwidth}
            \centering
            \begin{tikzpicture}
            \draw (0,0) -- (3.5,0);
            \draw (0,0) -- (0,-3.5);
            \draw (3.5,0) -- (3.5,-3.5);
            \draw (0,-3.5) -- (3.5,-3.5);
            \draw (0,-1.5) -- (2.5,-1.5);
            \draw (1.5,0) -- (1.5,-2.5);
            \draw (2.5,0) -- (2.5,-1);
            \draw (2.5,-1.5) -- (2.5,-3.5);
            \draw (1.5,-2.5) -- (3.5,-2.5);
            \draw (1.5,-1) -- (3.5,-1);
            \node at (0.75,-0.75) {A};
            \node at (2, -0.5) {B};
            \node at (3, -0.5) {C};
            \node at (2, -2) {D};
            \node at (3, -3) {D};
            \end{tikzpicture}
            \caption{The subsquare structure of $L$}
            \label{ab^2 L}
        \end{subfigure}
        
        \caption{}
    \end{figure}

    Fill the array $O$ as in \Cref{ab^2 outline}, where $x:y$ indicates that there are $y$ copies of symbol $x$ in that cell.

    \begin{figure}[h]
        \centering
        \renewcommand{\arraystretch}{2.7}
        \begin{tabular}{|l|l|l|l|}\hline
        \thead{$1: b(2b-a)$\\$2: b(a-b)$} & $1: b(a-b)$ & $3:b^2$ & $4:b^2$ \\\hline
        $1: b(a-b)$ & $2:(a-b)^2$ & $4:b(a-b)$ & $3:b(a-b)$ \\\hline
        $4:b^2$ & $3:b(a-b)$ & $1:b^2$ & \thead{$2:b(a-b)$\\$3:b(2b-a)$} \\\hline
        $3:b^2$ & $4:b(a-b)$ & \thead{$2:b(a-b)$\\$4:b(2b-a)$} & $1:b^2$\\\hline
        \end{tabular}
        \caption{The outline rectangle $O$}
        \label{ab^2 outline}
    \end{figure}

    The array $O$ satisfies all of the conditions of an outline rectangle when $N = a+2b$, so this array is the reduction of a latin square $L$ by \Cref{outline rectangle to square}. Clearly this array also has the cells $A_1$, $A_2$, $A_3$, $A_4$, $B'$, $C'$ and $D'$ filled as described above.

    Symbols 1 and 2 will become the symbols of $[a]$ in $L$, and since there are $a^2$ of them in $A_1\cup A_2\cup A_3\cup A_4$, $L$ will have a subsquare of order $a$, shown by $A$ in \Cref{ab^2 L}. Similarly, the entries of $B'$ force $B$ to be a subsquare of order $b$ on the symbols of $a + [b]$, $C'$ makes $C$ a subsquare of order $b$ on the symbols of $(a+b) + [b]$, and $D'$ causes $D$ to be a subsquare of order $b$ on the symbols of $[b]$. Therefore, due to the structure of $O$, any latin square $L$, constructed from $O$, will have the subsquare structure shown in \Cref{ab^2 L}.

    The latin square $L$ has subsquares $H_1,H_2,H_3$ on the main diagonal with $h_1 = a$, $h_2 = h_3 = b$ and symbols $1$ to $h_i$ in each subsquare $H_i$. So if $S_i = (h_1+\dots +h_{i-1}) + [h_i]$, then for all $(r,c)\in S_i\times S_i$ we have that the symbols in that subsquare are in $[h_i]$. Also, the same subsquares, but with disjoint symbols between subsquares, appear in the first $a$ or $b$ rows respectively. So if $(r,c)\in [h_i]\times S_i$ for some $i\in [3]$, then the symbols in that subsquare are in $S_i$.

    We construct a latin cube $C$ using Construction \ref{construction}, similar to the method used in \Cref{CubefromSquare}. To accomodate for the fact that a realization does not exist, the symbol-disjoint subsquares are now at the top of $L$, instead of on the main diagonal, but they are still column-disjoint. We know that $C$ must be a latin cube from our construction, so we now check that it has the required subcubes on the main diagonal.

    Note that for any part $h_i$, if $(r,c)\in S_i\times S_i$ then $L(r,c)\in [h_i]$ by the subsquare structure of $L$. Also, if $(r,c)\in [h_i]\times S_i$ then $L(r,c)\in S_i$ by the earlier assumption. So if $l,c,r\in S_i$ then $L(r,l) = a\in [h_i]$ and $L(a,c)\in S_i$. Thus, $C(r,c,l) = L(L(r,l),c)\in S_i$. Therefore, the cubes formed by the columns, rows and layers of $S_i$ are filled with only the elements $S_i$. So $C$ has the required subcubes.
\end{proof}

\begin{example}
    The latin squares in Figures \ref{ab^2 ex1} and \ref{ab^2 ex2} can be used to construct a $\LC(a^1b^2)$ with Construction \ref{construction}.

    \begin{figure}[h]
        \centering

        \begin{subfigure}[b]{0.3\textwidth}
            \centering
            $\arraycolsep=4pt\begin{array}{|cc|c|cc|cc|} \hline 
            1 & 1 & 1 & 3 & 3 & 4 & 4\\ 
            2 & 2 & 1 & 3 & 3 & 4 & 4\\ \hline
            1 & 1 & 2 & 4 & 4 & 3 & 3\\ \hline
            4 & 4 & 3 & 1 & 1 & 2 & 2\\ 
            4 & 4 & 3 & 1 & 1 & 3 & 3\\ \hline
            3 & 3 & 4 & 2 & 2 & 1 & 1\\ 
            3 & 3 & 4 & 4 & 4 & 1 & 1 \\ \hline\end{array}$
            \caption{The outline rectangle $O$}
        \end{subfigure}
        \begin{subfigure}[b]{0.3\textwidth}
            \centering
            $\arraycolsep=4pt\begin{array}{|ccc|cc|cc|} \hline
            1 & 3 & 2 & 4 & 5 & 6 & 7\\ 
            3 & 2 & 1 & 5 & 4 & 7 & 6\\ \hhline{|~~~|--|--|}
            2 & 1 & 3 & 6 & 7 & 4 & 5\\ \hline
            6 & \multicolumn{1}{c|}{7} & 4 & 1 & 2 & 5 & 3\\ 
            7 & \multicolumn{1}{c|}{6} & 5 & 2 & 1 & 3 & 4\\ \hline
            4 & \multicolumn{1}{c|}{5} & 6 & 7 & 3 & 1 & 2\\ 
            5 & \multicolumn{1}{c|}{4} & 7 & 3 & 6 & 2 & 1 \\ \hline\end{array}$
            \caption{The latin square $L$}
        \end{subfigure}
        
        \caption{A latin square $L$ for a $\LC(3^12^2)$ and the outline rectangle used to find $L$.}
        \label{ab^2 ex1}
    \end{figure}

    \begin{figure}[h]
        \centering

        \begin{subfigure}[b]{0.4\textwidth}
            \centering
            $$\arraycolsep=2pt\begin{array}{|cccc|ccc|ccc|}\hline 
            1 & 2 & 4 & 3 & 5 & 6 & 7 & 8 & 9 & 10\\ 
            2 & 4 & 3 & 1 & 6 & 7 & 5 & 9 & 10 & 8\\ 
            4 & 3 & 1 & 2 & 7 & 5 & 6 & 10 & 8 & 9\\ \hhline{|~~~~|---|---|}
            3 & 1 & 2 & 4 & 8 & 9 & 10 & 5 & 6 & 7\\ \hline
            8 & 9 & \multicolumn{1}{c|}{10} & 5 & 1 & 2 & 3 & 4 & 7 & 6\\ 
            9 & 10 & \multicolumn{1}{c|}{8} & 6 & 2 & 3 & 1 & 7 & 4 & 5\\ 
            10 & 8 & \multicolumn{1}{c|}{9} & 7 & 3 & 1 & 2 & 6 & 5 & 4\\ \hline
            5 & 6 & \multicolumn{1}{c|}{7} & 8 & 4 & 10 & 9 & 1 & 2 & 3\\ 
            6 & 7 & \multicolumn{1}{c|}{5} & 9 & 10 & 4 & 8 & 2 & 3 & 1\\ 
            7 & 5 & \multicolumn{1}{c|}{6} & 10 & 9 & 8 & 4 & 3 & 1 & 2\\ \hline \end{array}$$
            \caption{A latin square $L$ for a $\LC(4^13^2)$.}
        \end{subfigure}
        \begin{subfigure}[b]{0.4\textwidth}
            \centering
            $$\arraycolsep=2pt\begin{array}{|ccccc|ccc|ccc|} \hline
            1 & 4 & 5 & 2 & 3 & 6 & 7 & 8 & 9 & 10 & 11\\ 
            5 & 2 & 4 & 3 & 1 & 7 & 8 & 6 & 10 & 11 & 9\\ 
            4 & 5 & 3 & 1 & 2 & 8 & 6 & 7 & 11 & 9 & 10\\ \hhline{|~~~~~|---|---|}
            3 & 1 & 2 & 4 & 5 & 9 & 10 & 11 & 6 & 7 & 8\\ 
            2 & 3 & 1 & 5 & 4 & 10 & 11 & 9 & 7 & 8 & 6\\ \hline
            9 & 10 & \multicolumn{1}{c|}{11} & 6 & 7 & 1 & 2 & 3 & 8 & 4 & 5\\ 
            10 & 11 & \multicolumn{1}{c|}{9} & 7 & 8 & 2 & 3 & 1 & 5 & 6 & 4\\ 
            11 & 9 & \multicolumn{1}{c|}{10} & 8 & 6 & 3 & 1 & 2 & 4 & 5 & 7\\ \hline
            6 & 7 & \multicolumn{1}{c|}{8} & 9 & 10 & 11 & 4 & 5 & 1 & 2 & 3\\ 
            7 & 8 & \multicolumn{1}{c|}{6} & 10 & 11 & 5 & 9 & 4 & 2 & 3 & 1\\ 
            8 & 6 & \multicolumn{1}{c|}{7} & 11 & 9 & 4 & 5 & 10 & 3 & 1 & 2\\ \hline \end{array}$$
            \caption{A latin square $L$ for a $\LC(5^13^2)$.}
        \end{subfigure}
        
        \caption{A latin square $L$ which can be used to construct a $\LC(a^1b^2)$.}
        \label{ab^2 ex2}
    \end{figure}
    
\end{example}

We will now extend the previous lemma to any value of $n\geq 3$.

\begin{theorem}
\label{ab^n-1}
    For $n\geq 3$, a $\LC(ab^{n-1})$ exists if and only if $a\leq (n-1)b$.
\end{theorem}
\begin{proof}
    Suppose that a $\LC(ab^{n-1})$ exists with $a>b$. From \Cref{condition1} we may deduce that $a\leq (n-1)b$. 
    
    Now suppose that this condition is satisfied, and we will show that a $\LC(ab^{n-1})$ exists.

    \begin{itemize}
        \item If $n=3$ then we have already shown existence in \Cref{ab^2}.
        \item We know that a $\LS(ab^{n-1})$ exists for all $b<a\leq (n-2)b$ with $n\geq 4$ from Theorems \ref{small n squares} and \ref{squaresatmost2}. Thus, from \Cref{CubefromSquare} we have existence of the 3-realization.
        \item If $a=(n-1)b$ then existence has already been shown in \Cref{halving condition}.
    \end{itemize}
    
    So suppose that $(n-2)b < a <(n-1)b$.

    We prove existence by contructing a similar outline rectangle $O$ to the one in \Cref{ab^2} and then likewise constructing a latin square and cube from this. We will use the partitions 
    $$P=Q=R = (b^{n-2}(a-b(n-2))^1b^{n-1}),$$ 
    which means that our outline rectangle will be an $2(n-1)\times 2(n-1)$ array on $2(n-1)$ symbols, which we take to be $[2(n-1)]$. We now separate our outline rectangle into four sections, as shown in \Cref{ABCD}, each of which is an $(n-1)\times (n-1)$ array which can have multiple entries per cell.

    \begin{figure}[h]
        \centering

        \begin{tikzpicture}
            \draw (0,0) -- (1.5,0);
            \draw (0,0) -- (0,-1.5);
            \draw (1.5,0) -- (1.5,-1.5);
            \draw (0,-1.5) -- (1.5,-1.5);
            \draw (0,-0.75) -- (1.5,-0.75);
            \draw (0.75,0) -- (0.75,-1.5);
            \node at (0.375,-0.375) {A};
            \node at (1.125,-0.375) {B};
            \node at (0.375,-1.125) {C};
            \node at (1.125,-1.125) {D};
        \end{tikzpicture}
        \caption{}
        \label{ABCD}
    \end{figure}

    We now construct each of these arrays. For the array $D$, let entries of cell $(i,j)\in[n-1]\times[n-1]$ be defined as follows. All values modulo $n-1$ are considered with residues in $[n-1]$; that is, we consider $k(n-1)\mod n-1$ to be $n-1$ instead of $0$.
    \begin{itemize}
        \item $b^2$ copies of symbol $-i+j+1\mod n-1$ if $-i+j\not\equiv n-2\mod n-1$;
        \item $b(a-(n-2)b)$ copies of symbol $n-1$ and $b((n-1)b-a)$ copies of $n-1+i$ if $-i+j \equiv n-2\mod n-1$.
    \end{itemize}

    The array $D$ will be filled as in \Cref{D outline}, where $c=b(a-(n-2)b)$ and $d=b((n-1)b-a)$.

    \begin{figure}[h]
        \centering
        \renewcommand{\arraystretch}{2.7}
        \begin{tabular}{|llcll|}\hline
        $1: b^2$ & \multicolumn{1}{|l|}{$2: b^2$} & $\dots$ & \multicolumn{1}{|l|}{$n-2:b^2$} & \thead{$n-1: c$\\$n:d$} \\\hhline{|-|-|~|-|-|}
        \thead{$n-1: c$\\$n+1:d$} & \multicolumn{1}{|l|}{$1:b^2$} & $\dots$ & \multicolumn{1}{|l|}{$n-3: b^2$} & $n-2: b^2$ \\\hhline{|-|-|~|-|-|}
        \multicolumn{1}{|c}{$\vdots$} & \multicolumn{1}{c}{$\vdots$} & $\ddots$ & \multicolumn{1}{c}{$\vdots$} & \multicolumn{1}{c|}{$\vdots$} \\\hhline{|-|-|~|-|-|}
        $3:b^2$ & \multicolumn{1}{|l|}{$4:b(a-b)$} & $\dots$ & \multicolumn{1}{|l|}{$1:b^2$} & $2:b^2$ \\\hhline{|-|-|~|-|-|}
        $2:b^2$ & \multicolumn{1}{|l|}{$3:b(a-b)$} & $\dots$ & \multicolumn{1}{|l|}{\thead{$n-1: c$\\$2n-2:d$}} & $1:b^2$ \\\hline
        \end{tabular}
        \caption{The array $D$}
        \label{D outline}
    \end{figure}

    For array $B$ we use a similar construction. The entries of cell $(i,j)\in[n-1]\times[n-1]$ are
    \begin{itemize}
        \item $b^2$ copies of symbol $(-i+j+1\mod n-1)+n-1$ if $i\neq n-1$;
        \item $b(a-(n-2)b)$ copies of symbol $(-i+j+1\mod n-1)+n-1$ if $i = n-1$.
    \end{itemize}

    For array $C$, the entries of cell $(i,j)\in[n-1]\times[n-1]$ are
    \begin{itemize}
        \item $b^2$ copies of symbol $(i+j\mod n-1) + n-1$ if $j\neq n-1$;
        \item $b(a-(n-2)b)$ copies of symbol $(i+j\mod n-1) + n-1$ if $j = n-1$.
    \end{itemize}

    For array $A$, the entries of cell $(i,j)\in[n-1]\times[n-1]$ are
    \begin{itemize}
        \item $b(a-(n-2)b)$ copies of symbol $i+j\mod n-1$ if $i = n-1$ or $j=n-1$;
        \item $b^2$ copies of symbol $i+j\mod n-1$ if $i,j < n-1$ and $i+j > n-1$;
        \item $b(a-(n-2)b)$ copies of symbol $i+j$ and $b((n-1)b - a)$ copies of $i+j-1$ if $i+j \leq n-1$.
    \end{itemize}

    Combining these and ignoring repetitions, the outline rectangle will look like the array in \Cref{Gross Array}.

    \begin{figure}[ht]
    \begin{tabular}{|ccccc||ccccc|}
    \hline
    2,1 & \multicolumn{1}{|c|}{3,2} & $\dots$ & \multicolumn{1}{|c|}{$n-1,n-2$} & 1 & $n$ & \multicolumn{1}{|c|}{$n+1$} & $\dots$ & \multicolumn{1}{|c|}{$2n-3$} & $2n-2$\\\hhline{|-|-|~|-|-||-|-|~|-|-|}
    3,2 & \multicolumn{1}{|c|}{4,3} & $\dots$ & \multicolumn{1}{|c|}{1} & 2 & $2n-2$ & \multicolumn{1}{|c|}{$n$} & $\dots$ & \multicolumn{1}{|c|}{$2n-4$} & $2n-3$\\\hhline{|-|-|~|-|-||-|-|~|-|-|}
    $\vdots$ & $\vdots$ & $\ddots$ & $\vdots$ & $\vdots$ & $\vdots$ & $\vdots$ & $\ddots$ & $\vdots$ & $\vdots$\\\hhline{|-|-|~|-|-||-|-|~|-|-|}
    $n-1,n-2$ & \multicolumn{1}{|c|}{1} & $\dots$ & \multicolumn{1}{|c|}{$n-3$} & $n-2$ & $n+2$ & \multicolumn{1}{|c|}{$n+3$} & $\dots$ & \multicolumn{1}{|c|}{$n$} & $n+1$\\\hhline{|-|-|~|-|-||-|-|~|-|-|}
    1 & \multicolumn{1}{|c|}{2} & $\dots$ & \multicolumn{1}{|c|}{$n-2$} & $n-1$ & $n+1$ & \multicolumn{1}{|c|}{$n+2$} & $\dots$ & \multicolumn{1}{|c|}{$2n-2$} & $n$\\\hhline{:=:=:=:=:=::=:=:=:=:=:}
    $n+1$ & \multicolumn{1}{|c|}{$n+2$} & $\dots$ & \multicolumn{1}{|c|}{$2n-2$} & $n$ & 1 & \multicolumn{1}{|c|}{2} & $\dots$ & \multicolumn{1}{|c|}{$n-2$} & $n-1,n$\\\hhline{|-|-|~|-|-||-|-|~|-|-|}
    $n+2$ & \multicolumn{1}{|c|}{$n+3$} & $\dots$ & \multicolumn{1}{|c|}{$n$} & $n+1$ & $n-1,n+1$ & \multicolumn{1}{|c|}{1} & $\dots$ & \multicolumn{1}{|c|}{$n-3$} & $n-2$\\\hhline{|-|-|~|-|-||-|-|~|-|-|}
    $\vdots$ & $\vdots$ & $\ddots$ & $\vdots$ & $\vdots$ & $\vdots$ & $\vdots$ & $\ddots$ & $\vdots$ & $\vdots$\\\hhline{|-|-|~|-|-||-|-|~|-|-|}
    $2n-2$ & \multicolumn{1}{|c|}{$n$} & $\dots$ & \multicolumn{1}{|c|}{$2n-4$} & $2n-3$ & 3 & \multicolumn{1}{|c|}{4} & $\dots$ & \multicolumn{1}{|c|}{$1$} & 2\\\hhline{|-|-|~|-|-||-|-|~|-|-|}
    $n$ & \multicolumn{1}{|c|}{$n+1$} & $\dots$ & \multicolumn{1}{|c|}{$2n-3$} & $2n-2$ & 2 & \multicolumn{1}{|c|}{3} & $\dots$ & \multicolumn{1}{|c|}{$n-1,2n-2$} & 1\\\hline
    \end{tabular}
    \caption{The outline rectangle $O$}
    \label{Gross Array}
    \end{figure}

    The array $O$ satisfies all of the requirements for an outline rectangle, thus we can obtain a latin square $L$ from this array using \Cref{outline rectangle to square}. The subarray $A$ only contains symbols from $[n-1]$, and this symbol set becomes $[a]$ in $L$. Since $A$ contains $a^2$ of these symbols, the subarray $A$ becomes a subsquare of order $a$ in $L$.

    Considering the subarray $B$, the first row of cells each contain $b^2$ copies of a specified symbol. Each of these symbols is unique and represents a disjoint set of $b$ symbols in $L$. Thus, there are $n-1$ symbol-disjoint subsquares of order $b$ in the first $b$ rows of $B$, and they contain disjoint symbols to the subsquare of order $a$.

    In subarray $D$, the main diagonal cells contain only $b^2$ copies of symbol 1. In $L$, each of these becomes a subsquare of order $b$ on the symbols of $[b]$.

    The structure of $L$ is thus similar to a 2-realization of $(a^1b^{n-1})$, except the symbol-disjoint subsquares of order $b$ are moved to the top of the array and are replaced by subsquares with the symbols of $[b]$. For $i\in[n]$, where $h_1 = a$ and $h_i = b$ for $i\geq 2$, let $S_i = (h_1 + \dots + h_{i-1}) + [h_i]$. Then formally, for all $(r,c)\in S_i\times S_i$, we have $L(r,c)\in[h_i]$, and for all $(r,c)\in[h_i]\times S_i$, we have $L(r,c)\in S_i$. This is the same structure as the latin square $L$ in \Cref{ab^2}.

    Thus, we can use \Cref{construction} to build a latin cube $C$, and $C$ will have the required subcubes for a 3-realization by the same proof as given in \Cref{ab^2}.
    
\end{proof}

This completes the existence problem where $u=1$. We now consider the case when $u=2$. For 2-realizations, this case has the same bounds as for $u=1$.

\begin{lemma}
\label{(2b)^2b case}
    A $\LC((2b)^2b^1)$, a $\LC((3b)^2(2b)^1)$ and a $\LC((4b)^2(3b)^1)$ exist for all $b\geq 1$.
\end{lemma}
\begin{proof}
    Observe that a $\LC(2^21^1)$ is given in Example \ref{2^21^1}. To construct a $\LC((2b)^2b^1)$, inflate this latin cube with a latin cube of order $b$ using \Cref{inflation}.

    A $\LC(3^22^1)$ is given in \ref{the cube}. Use inflation to get a $\LC((3b)^2(2b)^1)$. The same process can be used for the $\LC(4^23^1)$ in \ref{the cube 2}
\end{proof}

\begin{lemma}
\label{a^2b^2}
    A $\LC(a^2b^2)$ exists for $a\leq 2b$.
\end{lemma}
\begin{proof}
    If $a=2b$ then take a $\LC(a^3)$ and replace one of the subcubes with a $\LC(b^2)$.

    Suppose that $a<2b$. Since we have that $a<2b<2a$, we know from \Cref{ab^n-1} that a $\LC((2b)^1a^2)$ exists. Take this latin cube and replace the subcube of order $2b$ with a $\LC(b^2)$. This is a $\LC(a^2b^2)$.
\end{proof}

This does not provide a complete solution, since there are partitions for which the 3-realization exists outside of this bound. For example, a $\LC(4^21^2)$ exists, since this can be constructed using a $\LC(4^22^1)$.

The following lemma provides a new upper bound for the case when $u=2$.

\begin{lemma}
\label{a2b(n-2) cond}
    If a $\LC(a^2b^{n-2})$ exists for $n\geq 3$, then $a\leq 2(n-2)b$.
\end{lemma}
\begin{proof}
    Firstly, for $n=3$, using the condition from \Cref{condition2} and $D = \{1,2\}$ for $h_1 = h_2 = a$, the condition simplifies to $b\geq a/2$. For $n=4$, taking $D = \{1,2,3\}$ in the condition from \Cref{condition2} gives that $a\leq 4b = 2(n-2)b$. Now consider the case where $n\geq 5$.
    
    We know that a latin square $\LS(a^2b^{n-2})$ exists for $a\leq (n-2)b$ where $n\geq 5$, so the corresponding latin cube also exists. So now we consider $a > (n-2)b$.

    Consider the outline rectangle in \Cref{a<=2(n-2)b figure}, which is an amalgamation of the layer $([N],[N],1)$ of a latin cube $\LC(a^2b^{n-2})$. $A$,$B$,$D$ and $G$ represent $a\times a$ subarrays, while $C$ and $F$ represent $(n-2)b\times a$ subarrays, and $E$ is an $a\times (n-2)b$ subarray. $A$ is within the subcube $H_1$ and $G$ is in the same files as the subcube $H_2$. We have amalgamated the symbols using the partition $(a^2((n-2)b)^1)$.

    \begin{figure}[h]
        \centering

        \begin{tikzpicture}
            \draw (0,0) -- (2,0);
            \draw (0,0) -- (0,-2);
            \draw (1.5,0) -- (1.5,-2);
            \draw (0,-1.5) -- (2,-1.5);
            \draw (0,-0.75) -- (2,-0.75);
            \draw (0.75,0) -- (0.75,-2);
            \draw (0,-2) -- (2,-2);
            \draw (2,0) -- (2,-2);
            \node at (0.375,-0.375) {A};
            \node at (1.125,-0.375) {D};
            \node at (0.375,-1.125) {B};
            \node at (1.125,-1.125) {G};
            \node at (0.375,-1.75) {C};
            \node at (1.125,-1.75) {F};
            \node at (1.75,-0.375) {E};
        \end{tikzpicture}
        \caption{}
        \label{a<=2(n-2)b figure}
    \end{figure}

    Let $x=(n-2)b$. Observe that there must be $ax$ of symbol 3 across $B$ and $C$, but there can be at most $x^2$ of this symbol in $C$. So there must be at least $(a-x)x$ of 3 in $B$ since $a > x$. We can repeat this for cells $D$ and $E$.

    Since there are at least $(a-x)x$ copies of 3 in $D$, there are at most $a^2 - (a-x)x$ of symbol 2 in $D$. Also, since $G$ is in the same files as the subcube $H_2$, $G$ cannot contain any copies of symbol 2. Thus, there must be at least $(a-x)x$ copies of symbol 2 in $F$. This is the same for cells $B$ and $C$.

    Therefore, there are at least $2(a-x)x$ copies of symbol 2 across $C$ and $F$. There can be at most $ax$ of this symbol across those cells, so $2(a-x)x\leq ax$. This simplifies to $a\leq 2x = 2(n-2)b$.
\end{proof}

For $n=3$ and $n=4$, this is the same bound that is achieved using the condition from \Cref{condition2}, by taking $D = \{1,2\}$ and $D = \{1,2,3\}$ respectively. For all greater values of $n$ this is a stricter bound than that condition. For example, the partition $(43^27^3)$ satisfies the conditions from Theorems \ref{condition1} and \ref{condition2}, but not \Cref{a2b(n-2) cond}. This partition is also not excluded by \Cref{halving condition}. Therefore, we can conclude that the three conditions together are still not sufficient for existence.

The existence of a $\LC(a^2b^{n-2})$ is incomplete. Latin squares provide existence for all $a\leq (n-2)b$, where $n\geq 5$, and \Cref{a^2b^2} gives the same result for $n=4$. The following lemma shows that if a $\LC(a^2b^1)$ exists for the entire range given in \Cref{a2b(n-2) cond}, then the remaining existence cases for a $\LC(a^2b^{n-2})$ follow immediately.

\begin{lemma}
    If a $\LC(a^2b^1)$ exists for all $a\leq 2b$ then a $\LC(a^2b^{n-2})$ exists for all $(n-2)b<a\leq 2(n-2)b$, where $n\geq 4$.
\end{lemma}
\begin{proof}
    Take any $a$ such that $(n-2)b<a\leq 2(n-2)b$. Then there exists a $\LC(a^2((n-2)b)^1)$. Replacing the subsquare of order $(n-2)b$ with a $\LC(b^{n-2})$ gives a $\LC(a^2b^{n-2})$.
\end{proof}

\section{$m$-dimensional hypercubes}

This idea of a realization of a partition can be extended to higher dimensions.

\begin{definition}
    An \emph{m-dimensional latin hypercube} $L = (L_{j_1,\dots,j_m})$ of order $N$ is an $m$-dimensional array of side $N$ on a symbol set of size $N$, such that for all $(j_1,\dots,j_m),(j'_1,\dots,j'_m)\in S^m$, and all $\alpha\in\{1,\dots,m\}$, if $j_i = j'_i$ for all $i\neq\alpha$ and $j_\alpha\neq j'_\alpha$ then $L_{j_1,\dots,j_m} \neq L_{j'_1,\dots,j'_m}$.
    
    As before, a \emph{subhypercube} of $L$ is an $m$-dimensional sub-array which is itself a latin hypercube.
\end{definition}

\begin{definition}
    For an integer partition $P = (h_1,h_2,\dots,h_n)$ of $N$ with $h_1\geq h_2\geq \dots\geq h_n > 0$, an \emph{m-realization} of $P$, denoted $m\RP(h_1h_2\dots h_n)$, is an $m$-dimensional latin hypercube of order $N$ with pairwise disjoint subhypercubes of order $h_i$ for each $i\in [n]$.
\end{definition}

\begin{theorem}
\label{combine}
    If there exists an $m\RP(h_1\dots h_n)$ and a $k\RP(h_1\dots h_n)$, then there exists an $(m+k-1)\RP(h_1\dots h_n)$.
\end{theorem}
\begin{proof}
    Let $F$ be a latin hypercube which gives an $m\RP(h_1\dots h_n)$ and let $G$ be a latin hypercube which gives a $k\RP(h_1\dots h_n)$. We will assume that these are in normal form and with symbols in increasing order across the subhypercubes.

    Let $f(x_1,\dots,x_m)$ be an entry of $F$ and $g(x_1,\dots,x_k)$ be an entry of $G$. For $l = k + m - 1$, we will construct an $l$-dimensional hypercube, $H$, by $$h(x_1,\dots,x_l) = f(g(x_1,\dots,x_k),x_{k+1},\dots,x_l).$$

    Suppose that $h(x_1,\dots,x_i,\dots,x_l) = h(x_1,\dots,x_i',\dots,x_l)$ for some $i\in[l]$. If $i>k$ then $$f(g(x_1,\dots,x_k),x_{k+1},\dots,x_i,\dots,x_l) = f(g(x_1,\dots,x_k),x_{k+1},\dots,x_i',\dots,x_l).$$ Since $F$ is a latin hypercube, we must have that $x_i = x_i'$.

    If $i\leq k$, then $$f(g(x_1,\dots,x_i,\dots,x_{k}),x_{k+1},\dots,x_l) = f(g(x_1,\dots,x_i',\dots,x_{k}),x_{k+1},\dots,x_l).$$ This gives that $g(x_1,\dots,x_i,\dots,x_{k}) = g(x_1,\dots,x_i',\dots,x_{k})$, and since $G$ is a latin hypercube, $x_i = x_i'$. Therefore $H$ is a latin hypercube.

    We now show that this $H$ contains the required subhypercubes. Let $S_i = (h_1 + \dots + h_{i-1}) + [h_i]$ for some part $h_i$. Given our earlier assumptions about $F$ and $G$, we know that if $(x_1,\dots,x_m)\in S_i\times S_i$ then $f(x_1,\dots,x_m)\in S_i$, and if $x_1,\dots,x_k\in S_i$ then $g(x_1,\dots,x_k)\in S_i$.

    So if $x_1,\dots,x_l\in S_i$ for some $i$, then $g(x_1,\dots,x_k) = a \in S_i$. Thus, $h(x_1,\dots,x_l) = f(a,x_{k+1},\dots,x_l)\in S_i$. Therefore there are subhypercubes found along the diagonal with the symbols in increasing order. Thus, $H$ is an $(m+k-1)\RP(h_1\dots h_n)$.
\end{proof}

We now prove an analogue of \Cref{CubefromSquare} for a hypercube of any dimension.

\begin{corollary}
    If there exists a $\LS(h_1\dots h_n)$, then there exists an $m\RP(h_1\dots h_n)$ for all $m\geq 3$.
\end{corollary}
\begin{proof}
    The proof for $m=3$ is given in \Cref{CubefromSquare}. Now suppose that the statement is true for some $m$.

    So we have a latin square $L$ which gives a $\LS(h_1\dots h_n)$ and a latin hypercube, $G$, which gives an $m\RP(h_1\dots h_n)$. By \Cref{combine}, there exists an $(m+1)\RP(h_1\dots h_n)$ since $2+m-1 = m+1$.

    Therefore, we have proven the result by induction.
\end{proof}

The question that naturally follows from this is whether the existence of an $m\RP(h_1\dots h_n)$ implies the existence of an $(m+1)\RP(h_1\dots h_n)$. This can be easily disproven in the general case (for $m\geq 3$) with the partition $(1^2)$. There is no 2-realization $\LS(1^2)$ but there is a $\LC(1^2)$, and it can be easily seen that there is no $4\RP(1^2)$. In fact, there is no $m\RP(1^2)$ for any even $m$.

The following theorem is a further generalisation, which starts with an $m$-realization.

\begin{corollary}
\label{inter-dimension thing}
    If there exists an $m\RP(h_1\dots h_n)$, then there exists a $k\RP(h_1\dots h_n)$ for all $k\geq m\geq 2$ with $k\equiv 1\pmod{m-1}$.
\end{corollary}
\begin{proof}
    If $k\equiv 1\pmod{m-1}$ then $k=k'(m-1) + 1$ for some $k'\in\mathbb{Z}^+$.

    When $k' = 1$, $k=m$ and the statement is true.

    Suppose that it is true for some $k'\geq 1$. Then there is a latin hypercube $F$ which gives an $m\RP(h_1\dots h_n)$ and a latin hypercube $G$ which gives a $k\RP(h_1\dots h_n)$. Since $l = (k'+1)(m-1) + 1 = k + m - 1$, there exists an $l\RP(h_1\dots h_n)$ by \Cref{combine}.
\end{proof}

\begin{lemma}
    An $m\RP(1^2)$ exists if and only if $m$ is even.
\end{lemma}
\begin{proof}
    Suppose such a hypercube exists and the entries are represented by $f(x_1,\dots,x_m)$. Without loss of generality, the subhypercubes must be filled as $f(1,\dots,1) = 1$ and $f(2,\dots,2) = 2$. For every cell with exactly one $x_i=2$, we know that $f(x_1,\dots,x_m)=2$ since it is in a line with the first subhypercube. Repeating this, for every cell of the hypercube, we get that
    $$f(x_1,\dots,x_m) = \begin{cases}
        1, & \text{if $x_i=2$ for an even number of $i$,}\\
        2, & \text{otherwise.}
    \end{cases}$$
    If $m$ is even, then the second subhypercube must contain $1$ by this process. Thus, $m$ must be odd.

    Since a $\LC(1^2)$ exists, we can use \Cref{inter-dimension thing} to find an $m\RP(1^2)$ for any odd $m$.
\end{proof}

We have shown that many partitions can be realized in latin cubes which cannot in latin squares. We note here that although these realizations cannot all be extended to a $4$-realization, there are some partitions which exist in $4$-dimensional hypercubes but not latin squares. One such hypercube is given in \ref{4D}.

\appendix
\section{}

\begin{append}
    \label{the cube}

The following arrays represent the layers of a $\LC(3^22^1)$.
$$\begin{matrix}
    \scalebox{0.9}{
    $\arraycolsep=4pt\begin{array}{|cccccccc|} \hline \cellcolor{lightgray}3 & \cellcolor{lightgray}1 & \cellcolor{lightgray}2 & 5 & 6 & 4 & 8 & 7\\
    \cellcolor{lightgray}1 & \cellcolor{lightgray}2 & \cellcolor{lightgray}3 & 8 & 4 & 5 & 7 & 6\\
    \cellcolor{lightgray}2 & \cellcolor{lightgray}3 & \cellcolor{lightgray}1 & 4 & 7 & 6 & 5 & 8\\ 4 & 8 & 6 & 2 & 3 & 7 & 1 & 5\\ 5 & 6 & 4 & 7 & 8 & 2 & 3 & 1\\ 6 & 5 & 7 & 3 & 1 & 8 & 2 & 4\\ 8 & 7 & 5 & 6 & 2 & 1 & 4 & 3\\ 7 & 4 & 8 & 1 & 5 & 3 & 6 & 2\\ \hline\end{array}$\quad
    $\arraycolsep=4pt\begin{array}{|cccccccc|} \hline \cellcolor{lightgray}2 & \cellcolor{lightgray}3 & \cellcolor{lightgray}1 & 6 & 4 & 8 & 7 & 5\\
    \cellcolor{lightgray}3 & \cellcolor{lightgray}1 & \cellcolor{lightgray}2 & 7 & 5 & 6 & 4 & 8\\
    \cellcolor{lightgray}1 & \cellcolor{lightgray}2 & \cellcolor{lightgray}3 & 5 & 6 & 4 & 8 & 7\\ 7 & 4 & 5 & 3 & 8 & 1 & 6 & 2\\ 4 & 5 & 8 & 1 & 7 & 3 & 2 & 6\\ 5 & 6 & 4 & 8 & 2 & 7 & 1 & 3\\ 6 & 8 & 7 & 2 & 1 & 5 & 3 & 4\\ 8 & 7 & 6 & 4 & 3 & 2 & 5 & 1\\ \hline \end{array}$\quad
    $\arraycolsep=4pt\begin{array}{|cccccccc|} \hline \cellcolor{lightgray}1 & \cellcolor{lightgray}2 & \cellcolor{lightgray}3 & 4 & 5 & 7 & 6 & 8\\
    \cellcolor{lightgray}2 & \cellcolor{lightgray}3 & \cellcolor{lightgray}1 & 5 & 6 & 4 & 8 & 7\\
    \cellcolor{lightgray}3 & \cellcolor{lightgray}1 & \cellcolor{lightgray}2 & 6 & 8 & 5 & 7 & 4\\ 5 & 6 & 4 & 8 & 7 & 2 & 3 & 1\\ 6 & 7 & 5 & 2 & 1 & 8 & 4 & 3\\ 8 & 4 & 6 & 7 & 3 & 1 & 5 & 2\\ 7 & 5 & 8 & 1 & 4 & 3 & 2 & 6\\ 4 & 8 & 7 & 3 & 2 & 6 & 1 & 5\\ \hline \end{array}$\quad
    $\arraycolsep=4pt\begin{array}{|cccccccc|} \hline 4 & 7 & 5 & 8 & 3 & 2 & 1 & 6\\ 5 & 8 & 4 & 2 & 1 & 7 & 6 & 3\\
    8 & 6 & 7 & 1 & 2 & 3 & 4 & 5\\
    1 & 3 & 8 & \cellcolor{lightgray}6 & \cellcolor{lightgray}5 & \cellcolor{lightgray}4 & 2 & 7\\
    7 & 1 & 3 & \cellcolor{lightgray}5 & \cellcolor{lightgray}4 & \cellcolor{lightgray}6 & 8 & 2\\
    3 & 2 & 1 & \cellcolor{lightgray}4 & \cellcolor{lightgray}6 & \cellcolor{lightgray}5 & 7 & 8\\
    2 & 4 & 6 & 3 & 7 & 8 & 5 & 1\\
    6 & 5 & 2 & 7 & 8 & 1 & 3 & 4\\ \hline \end{array}$}\\\\
    \scalebox{0.9}{\;$\arraycolsep=4pt\begin{array}{|cccccccc|} \hline 7 & 8 & 6 & 3 & 2 & 1 & 5 & 4\\
    6 & 4 & 7 & 1 & 8 & 3 & 2 & 5\\
    4 & 5 & 8 & 7 & 1 & 2 & 3 & 6\\
    3 & 2 & 1 & \cellcolor{lightgray}5 & \cellcolor{lightgray}4 & \cellcolor{lightgray}6 & 7 & 8\\
    8 & 3 & 2 & \cellcolor{lightgray}4 & \cellcolor{lightgray}6 & \cellcolor{lightgray}5 & 1 & 7\\
    2 & 7 & 3 & \cellcolor{lightgray}6 & \cellcolor{lightgray}5 & \cellcolor{lightgray}4 & 8 & 1\\
    5 & 1 & 4 & 8 & 3 & 7 & 6 & 2\\
    1 & 6 & 5 & 2 & 7 & 8 & 4 & 3\\ \hline \end{array}$\quad
    $\arraycolsep=4pt\begin{array}{|cccccccc|} \hline 6 & 5 & 8 & 1 & 7 & 3 & 4 & 2\\ 8 & 7 & 6 & 3 & 2 & 1 & 5 & 4\\
    7 & 4 & 5 & 2 & 3 & 8 & 6 & 1\\
    2 & 1 & 7 & \cellcolor{lightgray}4 & \cellcolor{lightgray}6 & \cellcolor{lightgray}5 & 8 & 3\\
    3 & 2 & 1 & \cellcolor{lightgray}6 & \cellcolor{lightgray}5 & \cellcolor{lightgray}4 & 7 & 8\\
    1 & 8 & 2 & \cellcolor{lightgray}5 & \cellcolor{lightgray}4 & \cellcolor{lightgray}6 & 3 & 7\\
    4 & 6 & 3 & 7 & 8 & 2 & 1 & 5\\
    5 & 3 & 4 & 8 & 1 & 7 & 2 & 6\\ \hline \end{array}$\quad
    $\arraycolsep=4pt\begin{array}{|cccccccc|} \hline 8 & 6 & 4 & 7 & 1 & 5 & 2 & 3\\ 4 & 5 & 8 & 6 & 7 & 2 & 3 & 1\\ 5 & 8 & 6 & 3 & 4 & 7 & 1 & 2\\ 6 & 7 & 3 & 1 & 2 & 8 & 5 & 4\\ 2 & 4 & 7 & 8 & 3 & 1 & 6 & 5\\ 7 & 1 & 5 & 2 & 8 & 3 & 4 & 6\\
    1 & 3 & 2 & 4 & 5 & 6 & \cellcolor{lightgray}7 & \cellcolor{lightgray}8\\
    3 & 2 & 1 & 5 & 6 & 4 & \cellcolor{lightgray}8 & \cellcolor{lightgray}7\\ \hline \end{array}$\quad
    $\arraycolsep=4pt\begin{array}{|cccccccc|} \hline 5 & 4 & 7 & 2 & 8 & 6 & 3 & 1\\ 7 & 6 & 5 & 4 & 3 & 8 & 1 & 2\\ 6 & 7 & 4 & 8 & 5 & 1 & 2 & 3\\ 8 & 5 & 2 & 7 & 1 & 3 & 4 & 6\\ 1 & 8 & 6 & 3 & 2 & 7 & 5 & 4\\ 4 & 3 & 8 & 1 & 7 & 2 & 6 & 5\\
    3 & 2 & 1 & 5 & 6 & 4 & \cellcolor{lightgray}8 & \cellcolor{lightgray}7\\
    2 & 1 & 3 & 6 & 4 & 5 & \cellcolor{lightgray}7 & \cellcolor{lightgray}8\\ \hline \end{array}$}
\end{matrix}$$
\end{append}

\begin{append}
\label{the cube 2}
The following arrays represent the layers of a $\LC(4^23^1)$.

\vspace{0.25cm}

\centering
    \scalebox{0.8}{
    $\arraycolsep=1pt\begin{array}{|ccccccccccc|} \hline
    \cellcolor{lightgray}1 & \cellcolor{lightgray}2 & \cellcolor{lightgray}3 & \cellcolor{lightgray}4 & 5 & 6 & 7 & 8 & 11 & 9 & 10\\
    \cellcolor{lightgray}2 & \cellcolor{lightgray}3 & \cellcolor{lightgray}4 & \cellcolor{lightgray}1 & 6 & 7 & 8 & 5 & 9 & 10 & 11\\
    \cellcolor{lightgray}3 & \cellcolor{lightgray}4 & \cellcolor{lightgray}1 & \cellcolor{lightgray}2 & 7 & 8 & 5 & 6 & 10 & 11 & 9\\
    \cellcolor{lightgray}4 & \cellcolor{lightgray}1 & \cellcolor{lightgray}2 & \cellcolor{lightgray}3 & 10 & 11 & 9 & 7 & 5 & 8 & 6\\
    5 & 6 & 7 & 10 & 11 & 4 & 3 & 9 & 8 & 2 & 1\\
    6 & 7 & 8 & 5 & 2 & 9 & 10 & 11 & 1 & 4 & 3\\
    7 & 8 & 5 & 11 & 9 & 10 & 1 & 4 & 3 & 6 & 2\\
    8 & 5 & 6 & 9 & 4 & 1 & 11 & 10 & 2 & 3 & 7\\
    11 & 9 & 10 & 8 & 3 & 5 & 2 & 1 & 6 & 7 & 4\\
    9 & 10 & 11 & 6 & 8 & 2 & 4 & 3 & 7 & 1 & 5\\
    10 & 11 & 9 & 7 & 1 & 3 & 6 & 2 & 4 & 5 & 8\\\hline
    \end{array}$\quad
    $\arraycolsep=1pt\begin{array}{|ccccccccccc|} \hline
\cellcolor{lightgray}4 & \cellcolor{lightgray}1 & \cellcolor{lightgray}2 & \cellcolor{lightgray}3 & 8 & 10 & 11 & 9 & 6 & 5 & 7\\
\cellcolor{lightgray}1 & \cellcolor{lightgray}2 & \cellcolor{lightgray}3 & \cellcolor{lightgray}4 & 5 & 6 & 7 & 8 & 11 & 9 & 10\\
\cellcolor{lightgray}2 & \cellcolor{lightgray}3 & \cellcolor{lightgray}4 & \cellcolor{lightgray}1 & 6 & 7 & 8 & 5 & 9 & 10 & 11\\
\cellcolor{lightgray}3 & \cellcolor{lightgray}4 & \cellcolor{lightgray}1 & \cellcolor{lightgray}2 & 7 & 8 & 5 & 6 & 10 & 11 & 9\\
9 & 5 & 6 & 7 & 10 & 1 & 2 & 11 & 3 & 4 & 8\\
10 & 6 & 7 & 8 & 9 & 11 & 1 & 4 & 5 & 3 & 2\\
6 & 7 & 8 & 5 & 11 & 3 & 9 & 10 & 2 & 1 & 4\\
11 & 8 & 5 & 6 & 1 & 9 & 10 & 2 & 4 & 7 & 3\\
5 & 11 & 9 & 10 & 2 & 4 & 6 & 3 & 7 & 8 & 1\\
7 & 9 & 10 & 11 & 4 & 5 & 3 & 1 & 8 & 2 & 6\\
8 & 10 & 11 & 9 & 3 & 2 & 4 & 7 & 1 & 6 & 5\\ \hline
\end{array}$\quad
    $\arraycolsep=1pt\begin{array}{|ccccccccccc|} \hline
\cellcolor{lightgray}3 & \cellcolor{lightgray}4 & \cellcolor{lightgray}1 & \cellcolor{lightgray}2 & 7 & 8 & 5 & 6 & 10 & 11 & 9\\
\cellcolor{lightgray}4 & \cellcolor{lightgray}1 & \cellcolor{lightgray}2 & \cellcolor{lightgray}3 & 9 & 5 & 10 & 11 & 7 & 6 & 8\\
\cellcolor{lightgray}1 & \cellcolor{lightgray}2 & \cellcolor{lightgray}3 & \cellcolor{lightgray}4 & 5 & 6 & 7 & 8 & 11 & 9 & 10\\
\cellcolor{lightgray}2 & \cellcolor{lightgray}3 & \cellcolor{lightgray}4 & \cellcolor{lightgray}1 & 6 & 7 & 8 & 5 & 9 & 10 & 11\\
7 & 11 & 5 & 6 & 3 & 2 & 9 & 10 & 1 & 8 & 4\\
8 & 9 & 6 & 7 & 11 & 10 & 2 & 3 & 4 & 1 & 5\\
5 & 10 & 7 & 8 & 1 & 9 & 11 & 2 & 6 & 4 & 3\\
6 & 7 & 8 & 5 & 10 & 11 & 4 & 9 & 3 & 2 & 1\\
10 & 6 & 11 & 9 & 4 & 3 & 1 & 7 & 8 & 5 & 2\\
11 & 8 & 9 & 10 & 2 & 1 & 6 & 4 & 5 & 3 & 7\\
9 & 5 & 10 & 11 & 8 & 4 & 3 & 1 & 2 & 7 & 6\\ \hline 
    \end{array}$\quad
    $\arraycolsep=1pt\begin{array}{|ccccccccccc|} \hline 
\cellcolor{lightgray}2 & \cellcolor{lightgray}3 & \cellcolor{lightgray}4 & \cellcolor{lightgray}1 & 6 & 7 & 8 & 5 & 9 & 10 & 11\\
\cellcolor{lightgray}3 & \cellcolor{lightgray}4 & \cellcolor{lightgray}1 & \cellcolor{lightgray}2 & 7 & 8 & 5 & 6 & 10 & 11 & 9\\
\cellcolor{lightgray}4 & \cellcolor{lightgray}1 & \cellcolor{lightgray}2 & \cellcolor{lightgray}3 & 11 & 9 & 6 & 10 & 8 & 7 & 5\\
\cellcolor{lightgray}1 & \cellcolor{lightgray}2 & \cellcolor{lightgray}3 & \cellcolor{lightgray}4 & 5 & 6 & 7 & 8 & 11 & 9 & 10\\
6 & 7 & 8 & 5 & 9 & 10 & 11 & 1 & 4 & 3 & 2\\
7 & 8 & 11 & 6 & 10 & 4 & 3 & 9 & 2 & 5 & 1\\
8 & 5 & 9 & 7 & 4 & 11 & 10 & 3 & 1 & 2 & 6\\
5 & 6 & 10 & 8 & 3 & 2 & 9 & 11 & 7 & 1 & 4\\
9 & 10 & 7 & 11 & 8 & 1 & 4 & 2 & 5 & 6 & 3\\
10 & 11 & 5 & 9 & 1 & 3 & 2 & 7 & 6 & 4 & 8\\
11 & 9 & 6 & 10 & 2 & 5 & 1 & 4 & 3 & 8 & 7\\ \hline 
    \end{array}$}

\vspace{0.25cm}

    \scalebox{0.8}{\;$\arraycolsep=1pt\begin{array}{|ccccccccccc|} \hline 
11 & 8 & 7 & 9 & 1 & 2 & 3 & 10 & 4 & 6 & 5\\
6 & 9 & 10 & 11 & 2 & 3 & 4 & 1 & 5 & 8 & 7\\
9 & 10 & 5 & 8 & 3 & 4 & 1 & 11 & 7 & 2 & 6\\
8 & 5 & 11 & 10 & 4 & 1 & 2 & 9 & 6 & 7 & 3\\
1 & 2 & 3 & 4 & \cellcolor{lightgray}5 & \cellcolor{lightgray}6 & \cellcolor{lightgray}7 & \cellcolor{lightgray}8 & 11 & 9 & 10\\
2 & 3 & 4 & 1 & \cellcolor{lightgray}6 & \cellcolor{lightgray}7 & \cellcolor{lightgray}8 & \cellcolor{lightgray}5 & 9 & 10 & 11\\
3 & 4 & 1 & 2 & \cellcolor{lightgray}7 & \cellcolor{lightgray}8 & \cellcolor{lightgray}5 & \cellcolor{lightgray}6 & 10 & 11 & 9\\
10 & 11 & 9 & 3 & \cellcolor{lightgray}8 & \cellcolor{lightgray}5 & \cellcolor{lightgray}6 & \cellcolor{lightgray}7 & 1 & 4 & 2\\
7 & 1 & 6 & 5 & 11 & 9 & 10 & 4 & 2 & 3 & 8\\
4 & 6 & 8 & 7 & 9 & 10 & 11 & 2 & 3 & 5 & 1\\
5 & 7 & 2 & 6 & 10 & 11 & 9 & 3 & 8 & 1 & 4\\ \hline 
    \end{array}$\quad
    $\arraycolsep=1pt\begin{array}{|ccccccccccc|} \hline 
10 & 5 & 6 & 11 & 9 & 1 & 2 & 3 & 7 & 8 & 4\\
9 & 11 & 5 & 8 & 10 & 2 & 3 & 4 & 1 & 7 & 6\\
11 & 7 & 9 & 10 & 2 & 3 & 4 & 1 & 6 & 5 & 8\\
5 & 9 & 10 & 6 & 11 & 4 & 1 & 2 & 8 & 3 & 7\\
4 & 10 & 11 & 9 & \cellcolor{lightgray}8 & \cellcolor{lightgray}5 & \cellcolor{lightgray}6 & \cellcolor{lightgray}7 & 2 & 1 & 3\\
1 & 2 & 3 & 4 & \cellcolor{lightgray}5 & \cellcolor{lightgray}6 & \cellcolor{lightgray}7 & \cellcolor{lightgray}8 & 11 & 9 & 10\\
2 & 3 & 4 & 1 & \cellcolor{lightgray}6 & \cellcolor{lightgray}7 & \cellcolor{lightgray}8 & \cellcolor{lightgray}5 & 9 & 10 & 11\\
3 & 4 & 1 & 2 & \cellcolor{lightgray}7 & \cellcolor{lightgray}8 & \cellcolor{lightgray}5 & \cellcolor{lightgray}6 & 10 & 11 & 9\\
6 & 8 & 2 & 7 & 1 & 11 & 9 & 10 & 3 & 4 & 5\\
8 & 1 & 7 & 5 & 3 & 9 & 10 & 11 & 4 & 6 & 2\\
7 & 6 & 8 & 3 & 4 & 10 & 11 & 9 & 5 & 2 & 1\\ \hline 
    \end{array}$\quad
    $\arraycolsep=1pt\begin{array}{|ccccccccccc|} \hline 
7 & 6 & 9 & 10 & 3 & 11 & 1 & 2 & 5 & 4 & 8\\
11 & 10 & 6 & 7 & 4 & 9 & 2 & 3 & 8 & 5 & 1\\
5 & 9 & 11 & 6 & 1 & 10 & 3 & 4 & 2 & 8 & 7\\
10 & 11 & 8 & 9 & 2 & 3 & 4 & 1 & 7 & 6 & 5\\
3 & 4 & 1 & 2 & \cellcolor{lightgray}7 & \cellcolor{lightgray}8 & \cellcolor{lightgray}5 & \cellcolor{lightgray}6 & 10 & 11 & 9\\
9 & 1 & 10 & 11 & \cellcolor{lightgray}8 & \cellcolor{lightgray}5 & \cellcolor{lightgray}6 & \cellcolor{lightgray}7 & 3 & 2 & 4\\
1 & 2 & 3 & 4 & \cellcolor{lightgray}5 & \cellcolor{lightgray}6 & \cellcolor{lightgray}7 & \cellcolor{lightgray}8 & 11 & 9 & 10\\
2 & 3 & 4 & 1 & \cellcolor{lightgray}6 & \cellcolor{lightgray}7 & \cellcolor{lightgray}8 & \cellcolor{lightgray}5 & 9 & 10 & 11\\
8 & 7 & 5 & 3 & 10 & 2 & 11 & 9 & 4 & 1 & 6\\
6 & 5 & 2 & 8 & 11 & 4 & 9 & 10 & 1 & 7 & 3\\
4 & 8 & 7 & 5 & 9 & 1 & 10 & 11 & 6 & 3 & 2\\ \hline 
    \end{array}$\quad
    $\arraycolsep=1pt\begin{array}{|ccccccccccc|} \hline 
9 & 10 & 11 & 5 & 2 & 3 & 4 & 1 & 8 & 7 & 6\\
10 & 8 & 7 & 9 & 3 & 4 & 11 & 2 & 6 & 1 & 5\\
8 & 11 & 10 & 7 & 4 & 1 & 9 & 3 & 5 & 6 & 2\\
7 & 6 & 9 & 11 & 1 & 2 & 10 & 4 & 3 & 5 & 8\\
2 & 3 & 4 & 1 & \cellcolor{lightgray}6 & \cellcolor{lightgray}7 & \cellcolor{lightgray}8 & \cellcolor{lightgray}5 & 9 & 10 & 11\\
3 & 4 & 1 & 2 & \cellcolor{lightgray}7 & \cellcolor{lightgray}8 & \cellcolor{lightgray}5 & \cellcolor{lightgray}6 & 10 & 11 & 9\\
11 & 9 & 2 & 10 & \cellcolor{lightgray}8 & \cellcolor{lightgray}5 & \cellcolor{lightgray}6 & \cellcolor{lightgray}7 & 4 & 3 & 1\\
1 & 2 & 3 & 4 & \cellcolor{lightgray}5 & \cellcolor{lightgray}6 & \cellcolor{lightgray}7 & \cellcolor{lightgray}8 & 11 & 9 & 10\\
4 & 5 & 8 & 6 & 9 & 10 & 3 & 11 & 1 & 2 & 7\\
5 & 7 & 6 & 3 & 10 & 11 & 1 & 9 & 2 & 8 & 4\\
6 & 1 & 5 & 8 & 11 & 9 & 2 & 10 & 7 & 4 & 3\\ \hline 
    \end{array}$}

\vspace{0.25cm}
    
    \scalebox{0.8}{\;$\arraycolsep=1pt\begin{array}{|ccccccccccc|} \hline 
5 & 9 & 8 & 7 & 10 & 4 & 6 & 11 & 1 & 3 & 2\\
8 & 6 & 9 & 5 & 11 & 10 & 1 & 7 & 4 & 2 & 3\\
6 & 5 & 7 & 9 & 8 & 11 & 10 & 2 & 3 & 4 & 1\\
9 & 7 & 6 & 8 & 3 & 5 & 11 & 10 & 2 & 1 & 4\\
10 & 8 & 2 & 11 & 1 & 9 & 4 & 3 & 7 & 6 & 5\\
11 & 10 & 5 & 3 & 4 & 2 & 9 & 1 & 6 & 7 & 8\\
4 & 11 & 10 & 6 & 2 & 1 & 3 & 9 & 8 & 5 & 7\\
7 & 1 & 11 & 10 & 9 & 3 & 2 & 4 & 5 & 8 & 6\\
3 & 4 & 1 & 2 & 7 & 6 & 5 & 8 & \cellcolor{lightgray}10 & \cellcolor{lightgray}11 & \cellcolor{lightgray}9\\
2 & 3 & 4 & 1 & 6 & 7 & 8 & 5 & \cellcolor{lightgray}9 & \cellcolor{lightgray}10 & \cellcolor{lightgray}11\\
1 & 2 & 3 & 4 & 5 & 8 & 7 & 6 & \cellcolor{lightgray}11 & \cellcolor{lightgray}9 & \cellcolor{lightgray}10\\ \hline 
    \end{array}$\quad
    $\arraycolsep=1pt\begin{array}{|ccccccccccc|} \hline 
6 & 11 & 5 & 8 & 4 & 9 & 10 & 7 & 2 & 1 & 3\\
5 & 7 & 11 & 6 & 8 & 1 & 9 & 10 & 3 & 4 & 2\\
7 & 6 & 8 & 11 & 10 & 5 & 2 & 9 & 1 & 3 & 4\\
11 & 8 & 7 & 5 & 9 & 10 & 6 & 3 & 4 & 2 & 1\\
8 & 9 & 10 & 3 & 2 & 11 & 1 & 4 & 5 & 7 & 6\\
4 & 5 & 9 & 10 & 1 & 3 & 11 & 2 & 8 & 6 & 7\\
10 & 1 & 6 & 9 & 3 & 2 & 4 & 11 & 7 & 8 & 5\\
9 & 10 & 2 & 7 & 11 & 4 & 3 & 1 & 6 & 5 & 8\\
2 & 3 & 4 & 1 & 6 & 7 & 8 & 5 & \cellcolor{lightgray}11 & \cellcolor{lightgray}9 & \cellcolor{lightgray}10\\
1 & 2 & 3 & 4 & 5 & 8 & 7 & 6 & \cellcolor{lightgray}10 & \cellcolor{lightgray}11 & \cellcolor{lightgray}9\\
3 & 4 & 1 & 2 & 7 & 6 & 5 & 8 & \cellcolor{lightgray}9 & \cellcolor{lightgray}10 & \cellcolor{lightgray}11\\ \hline 
    \end{array}$\quad
    $\arraycolsep=1pt\begin{array}{|ccccccccccc|} \hline 
8 & 7 & 10 & 6 & 11 & 5 & 9 & 4 & 3 & 2 & 1\\
7 & 5 & 8 & 10 & 1 & 11 & 6 & 9 & 2 & 3 & 4\\
10 & 8 & 6 & 5 & 9 & 2 & 11 & 7 & 4 & 1 & 3\\
6 & 10 & 5 & 7 & 8 & 9 & 3 & 11 & 1 & 4 & 2\\
11 & 1 & 9 & 8 & 4 & 3 & 10 & 2 & 6 & 5 & 7\\
5 & 11 & 2 & 9 & 3 & 1 & 4 & 10 & 7 & 8 & 6\\
9 & 6 & 11 & 3 & 10 & 4 & 2 & 1 & 5 & 7 & 8\\
4 & 9 & 7 & 11 & 2 & 10 & 1 & 3 & 8 & 6 & 5\\
1 & 2 & 3 & 4 & 5 & 8 & 7 & 6 & \cellcolor{lightgray}9 & \cellcolor{lightgray}10 & \cellcolor{lightgray}11\\
3 & 4 & 1 & 2 & 7 & 6 & 5 & 8 & \cellcolor{lightgray}11 & \cellcolor{lightgray}9 & \cellcolor{lightgray}10\\
2 & 3 & 4 & 1 & 6 & 7 & 8 & 5 & \cellcolor{lightgray}10 & \cellcolor{lightgray}11 & \cellcolor{lightgray}9\\ \hline 
    \end{array}$}

\end{append}

\begin{append}
\label{4D}

The six arrays in each row form a latin cube, and when all six latin cubes are combined gives a latin hypercube with a realization of $(2^21^2)$.

\vspace{0.25cm}

\centering
$\arraycolsep=3pt\begin{array}{|cccccc|} \hline
\cellcolor{lightgray}1 & \cellcolor{lightgray}2 & 5 & 6 & 3 & 4\\
\cellcolor{lightgray}2 & \cellcolor{lightgray}1 & 6 & 5 & 4 & 3\\
3 & 6 & 2 & 4 & 5 & 1\\
4 & 5 & 1 & 3 & 6 & 2\\
6 & 3 & 4 & 2 & 1 & 5\\
5 & 4 & 3 & 1 & 2 & 6\\
\hline \end{array}$\quad
$\arraycolsep=3pt\begin{array}{|cccccc|} \hline
\cellcolor{lightgray}2 & \cellcolor{lightgray}1 & 6 & 5 & 4 & 3\\
\cellcolor{lightgray}1 & \cellcolor{lightgray}2 & 5 & 6 & 3 & 4\\
5 & 4 & 3 & 1 & 2 & 6\\
6 & 3 & 4 & 2 & 1 & 5\\
4 & 5 & 1 & 3 & 6 & 2\\
3 & 6 & 2 & 4 & 5 & 1\\
\hline \end{array}$\quad
$\arraycolsep=3pt\begin{array}{|cccccc|} \hline
6 & 5 & 4 & 3 & 1 & 2\\
3 & 4 & 2 & 1 & 5 & 6\\
4 & 3 & 1 & 2 & 6 & 5\\
1 & 2 & 6 & 5 & 3 & 4\\
5 & 6 & 3 & 4 & 2 & 1\\
2 & 1 & 5 & 6 & 4 & 3\\
\hline \end{array}$\quad
$\arraycolsep=3pt\begin{array}{|cccccc|} \hline
5 & 6 & 3 & 4 & 2 & 1\\
4 & 3 & 1 & 2 & 6 & 5\\
2 & 1 & 5 & 6 & 3 & 4\\
3 & 4 & 2 & 1 & 5 & 6\\
1 & 2 & 6 & 5 & 4 & 3\\
6 & 5 & 4 & 3 & 1 & 2\\
\hline \end{array}$\quad
$\arraycolsep=3pt\begin{array}{|cccccc|} \hline
3 & 4 & 2 & 1 & 5 & 6\\
6 & 5 & 4 & 3 & 1 & 2\\
1 & 2 & 6 & 5 & 4 & 3\\
5 & 6 & 3 & 4 & 2 & 1\\
2 & 1 & 5 & 6 & 3 & 4\\
4 & 3 & 1 & 2 & 6 & 5\\
\hline \end{array}$\quad
$\arraycolsep=3pt\begin{array}{|cccccc|} \hline
4 & 3 & 1 & 2 & 6 & 5\\
5 & 6 & 3 & 4 & 2 & 1\\
6 & 5 & 4 & 3 & 1 & 2\\
2 & 1 & 5 & 6 & 4 & 3\\
3 & 4 & 2 & 1 & 5 & 6\\
1 & 2 & 6 & 5 & 3 & 4\\
\hline \end{array}$\quad

\vspace{0.2cm}
\hrule
\vspace{0.2cm}

$\arraycolsep=3pt\begin{array}{|cccccc|} \hline
\cellcolor{lightgray}2 & \cellcolor{lightgray}1 & 6 & 5 & 4 & 3\\
\cellcolor{lightgray}\cellcolor{lightgray}1 & \cellcolor{lightgray}2 & 5 & 6 & 3 & 4\\
4 & 5 & 1 & 3 & 6 & 2\\
3 & 4 & 2 & 1 & 5 & 6\\
5 & 6 & 3 & 4 & 2 & 1\\
6 & 3 & 4 & 2 & 1 & 5\\
\hline \end{array}$\quad
$\arraycolsep=3pt\begin{array}{|cccccc|} \hline
\cellcolor{lightgray}1 & \cellcolor{lightgray}2 & 5 & 6 & 3 & 4\\
\cellcolor{lightgray}2 & \cellcolor{lightgray}1 & 6 & 5 & 4 & 3\\
6 & 3 & 4 & 2 & 1 & 5\\
5 & 6 & 3 & 4 & 2 & 1\\
3 & 4 & 2 & 1 & 5 & 6\\
4 & 5 & 1 & 3 & 6 & 2\\
\hline \end{array}$\quad
$\arraycolsep=3pt\begin{array}{|cccccc|} \hline
4 & 3 & 1 & 2 & 6 & 5\\
5 & 6 & 3 & 4 & 2 & 1\\
3 & 4 & 2 & 1 & 5 & 6\\
2 & 1 & 5 & 6 & 4 & 3\\
6 & 5 & 4 & 3 & 1 & 2\\
1 & 2 & 6 & 5 & 3 & 4\\
\hline \end{array}$\quad
$\arraycolsep=3pt\begin{array}{|cccccc|} \hline
3 & 4 & 2 & 1 & 5 & 6\\
6 & 5 & 4 & 3 & 1 & 2\\
1 & 2 & 6 & 5 & 4 & 3\\
4 & 3 & 1 & 2 & 6 & 5\\
2 & 1 & 5 & 6 & 3 & 4\\
5 & 6 & 3 & 4 & 2 & 1\\
\hline \end{array}$\quad
$\arraycolsep=3pt\begin{array}{|cccccc|} \hline
5 & 6 & 3 & 4 & 2 & 1\\
4 & 3 & 1 & 2 & 6 & 5\\
2 & 1 & 5 & 6 & 3 & 4\\
6 & 5 & 4 & 3 & 1 & 2\\
1 & 2 & 6 & 5 & 4 & 3\\
3 & 4 & 2 & 1 & 5 & 6\\
\hline \end{array}$\quad
$\arraycolsep=3pt\begin{array}{|cccccc|} \hline
6 & 5 & 4 & 3 & 1 & 2\\
3 & 4 & 2 & 1 & 5 & 6\\
5 & 6 & 3 & 4 & 2 & 1\\
1 & 2 & 6 & 5 & 3 & 4\\
4 & 3 & 1 & 2 & 6 & 5\\
2 & 1 & 5 & 6 & 4 & 3\\
\hline \end{array}$\quad

\vspace{0.2cm}
\hrule
\vspace{0.2cm}

$\arraycolsep=3pt\begin{array}{|cccccc|} \hline
5 & 6 & 4 & 3 & 1 & 2\\
6 & 5 & 3 & 4 & 2 & 1\\
1 & 3 & 6 & 2 & 4 & 5\\
2 & 1 & 5 & 6 & 3 & 4\\
3 & 4 & 2 & 1 & 5 & 6\\
4 & 2 & 1 & 5 & 6 & 3\\
\hline \end{array}$\quad
$\arraycolsep=3pt\begin{array}{|cccccc|} \hline
6 & 5 & 3 & 4 & 2 & 1\\
5 & 6 & 4 & 3 & 1 & 2\\
4 & 2 & 1 & 5 & 6 & 3\\
3 & 4 & 2 & 1 & 5 & 6\\
2 & 1 & 5 & 6 & 3 & 4\\
1 & 3 & 6 & 2 & 4 & 5\\
\hline \end{array}$\quad
$\arraycolsep=3pt\begin{array}{|cccccc|} \hline
2 & 1 & 5 & 6 & 3 & 4\\
4 & 3 & 1 & 2 & 6 & 5\\
6 & 5 & \cellcolor{lightgray}3 & \cellcolor{lightgray}4 & 1 & 2\\
5 & 6 & \cellcolor{lightgray}4 & \cellcolor{lightgray}3 & 2 & 1\\
1 & 2 & 6 & 5 & 4 & 3\\
3 & 4 & 2 & 1 & 5 & 6\\
\hline \end{array}$\quad
$\arraycolsep=3pt\begin{array}{|cccccc|} \hline
1 & 2 & 6 & 5 & 4 & 3\\
3 & 4 & 2 & 1 & 5 & 6\\
5 & 6 & \cellcolor{lightgray}4 & \cellcolor{lightgray}3 & 2 & 1\\
6 & 5 & \cellcolor{lightgray}3 & \cellcolor{lightgray}4 & 1 & 2\\
4 & 3 & 1 & 2 & 6 & 5\\
2 & 1 & 5 & 6 & 3 & 4\\
\hline \end{array}$\quad
$\arraycolsep=3pt\begin{array}{|cccccc|} \hline
4 & 3 & 1 & 2 & 6 & 5\\
2 & 1 & 5 & 6 & 3 & 4\\
3 & 4 & 2 & 1 & 5 & 6\\
1 & 2 & 6 & 5 & 4 & 3\\
5 & 6 & 3 & 4 & 1 & 2\\
6 & 5 & 4 & 3 & 2 & 1\\
\hline \end{array}$\quad
$\arraycolsep=3pt\begin{array}{|cccccc|} \hline
3 & 4 & 2 & 1 & 5 & 6\\
1 & 2 & 6 & 5 & 4 & 3\\
2 & 1 & 5 & 6 & 3 & 4\\
4 & 3 & 1 & 2 & 6 & 5\\
6 & 5 & 4 & 3 & 2 & 1\\
5 & 6 & 3 & 4 & 1 & 2\\
\hline \end{array}$\quad

\vspace{0.2cm}
\hrule
\vspace{0.2cm}

$\arraycolsep=3pt\begin{array}{|cccccc|} \hline
6 & 5 & 3 & 4 & 2 & 1\\
5 & 6 & 4 & 3 & 1 & 2\\
2 & 4 & 5 & 1 & 3 & 6\\
1 & 3 & 6 & 2 & 4 & 5\\
4 & 2 & 1 & 5 & 6 & 3\\
3 & 1 & 2 & 6 & 5 & 4\\
\hline \end{array}$\quad
$\arraycolsep=3pt\begin{array}{|cccccc|} \hline
5 & 6 & 4 & 3 & 1 & 2\\
6 & 5 & 3 & 4 & 2 & 1\\
3 & 1 & 2 & 6 & 5 & 4\\
4 & 2 & 1 & 5 & 6 & 3\\
1 & 3 & 6 & 2 & 4 & 5\\
2 & 4 & 5 & 1 & 3 & 6\\
\hline \end{array}$\quad
$\arraycolsep=3pt\begin{array}{|cccccc|} \hline
3 & 4 & 2 & 1 & 5 & 6\\
1 & 2 & 6 & 5 & 4 & 3\\
5 & 6 & \cellcolor{lightgray}4 & \cellcolor{lightgray}3 & 2 & 1\\
6 & 5 & \cellcolor{lightgray}3 & \cellcolor{lightgray}4 & 1 & 2\\
2 & 1 & 5 & 6 & 3 & 4\\
4 & 3 & 1 & 2 & 6 & 5\\
\hline \end{array}$\quad
$\arraycolsep=3pt\begin{array}{|cccccc|} \hline
4 & 3 & 1 & 2 & 6 & 5\\
2 & 1 & 5 & 6 & 3 & 4\\
6 & 5 & \cellcolor{lightgray}3 & \cellcolor{lightgray}4 & 1 & 2\\
5 & 6 & \cellcolor{lightgray}4 & \cellcolor{lightgray}3 & 2 & 1\\
3 & 4 & 2 & 1 & 5 & 6\\
1 & 2 & 6 & 5 & 4 & 3\\
\hline \end{array}$\quad
$\arraycolsep=3pt\begin{array}{|cccccc|} \hline
1 & 2 & 6 & 5 & 4 & 3\\
3 & 4 & 2 & 1 & 5 & 6\\
4 & 3 & 1 & 2 & 6 & 5\\
2 & 1 & 5 & 6 & 3 & 4\\
6 & 5 & 4 & 3 & 2 & 1\\
5 & 6 & 3 & 4 & 1 & 2\\
\hline \end{array}$\quad
$\arraycolsep=3pt\begin{array}{|cccccc|} \hline
2 & 1 & 5 & 6 & 3 & 4\\
4 & 3 & 1 & 2 & 6 & 5\\
1 & 2 & 6 & 5 & 4 & 3\\
3 & 4 & 2 & 1 & 5 & 6\\
5 & 6 & 3 & 4 & 1 & 2\\
6 & 5 & 4 & 3 & 2 & 1\\
\hline \end{array}$\quad

\vspace{0.2cm}
\hrule
\vspace{0.2cm}

$\arraycolsep=3pt\begin{array}{|cccccc|} \hline
3 & 4 & 1 & 2 & 5 & 6\\
4 & 3 & 2 & 1 & 6 & 5\\
6 & 2 & 4 & 5 & 1 & 3\\
5 & 6 & 3 & 4 & 2 & 1\\
2 & 1 & 5 & 6 & 3 & 4\\
1 & 5 & 6 & 3 & 4 & 2\\
\hline \end{array}$\quad
$\arraycolsep=3pt\begin{array}{|cccccc|} \hline
4 & 3 & 2 & 1 & 6 & 5\\
3 & 4 & 1 & 2 & 5 & 6\\
1 & 5 & 6 & 3 & 4 & 2\\
2 & 1 & 5 & 6 & 3 & 4\\
5 & 6 & 3 & 4 & 2 & 1\\
6 & 2 & 4 & 5 & 1 & 3\\
\hline \end{array}$\quad
$\arraycolsep=3pt\begin{array}{|cccccc|} \hline
1 & 2 & 6 & 5 & 4 & 3\\
6 & 5 & 4 & 3 & 1 & 2\\
2 & 1 & 5 & 6 & 3 & 4\\
4 & 3 & 1 & 2 & 5 & 6\\
3 & 4 & 2 & 1 & 6 & 5\\
5 & 6 & 3 & 4 & 2 & 1\\
\hline \end{array}$\quad
$\arraycolsep=3pt\begin{array}{|cccccc|} \hline
2 & 1 & 5 & 6 & 3 & 4\\
5 & 6 & 3 & 4 & 2 & 1\\
3 & 4 & 1 & 2 & 6 & 5\\
1 & 2 & 6 & 5 & 4 & 3\\
6 & 5 & 4 & 3 & 1 & 2\\
4 & 3 & 2 & 1 & 5 & 6\\
\hline \end{array}$\quad
$\arraycolsep=3pt\begin{array}{|cccccc|} \hline
6 & 5 & 4 & 3 & 1 & 2\\
1 & 2 & 6 & 5 & 4 & 3\\
5 & 6 & 3 & 4 & 2 & 1\\
3 & 4 & 2 & 1 & 6 & 5\\
4 & 3 & 1 & 2 & \cellcolor{lightgray}5 & 6\\
2 & 1 & 5 & 6 & 3 & 4\\
\hline \end{array}$\quad
$\arraycolsep=3pt\begin{array}{|cccccc|} \hline
5 & 6 & 3 & 4 & 2 & 1\\
2 & 1 & 5 & 6 & 3 & 4\\
4 & 3 & 2 & 1 & 5 & 6\\
6 & 5 & 4 & 3 & 1 & 2\\
1 & 2 & 6 & 5 & 4 & 3\\
3 & 4 & 1 & 2 & 6 & 5\\
\hline \end{array}$\quad

\vspace{0.2cm}
\hrule
\vspace{0.2cm}

$\arraycolsep=3pt\begin{array}{|cccccc|} \hline
4 & 3 & 2 & 1 & 6 & 5\\
3 & 4 & 1 & 2 & 5 & 6\\
5 & 1 & 3 & 6 & 2 & 4\\
6 & 2 & 4 & 5 & 1 & 3\\
1 & 5 & 6 & 3 & 4 & 2\\
2 & 6 & 5 & 4 & 3 & 1\\
\hline \end{array}$\quad
$\arraycolsep=3pt\begin{array}{|cccccc|} \hline
3 & 4 & 1 & 2 & 5 & 6\\
4 & 3 & 2 & 1 & 6 & 5\\
2 & 6 & 5 & 4 & 3 & 1\\
1 & 5 & 6 & 3 & 4 & 2\\
6 & 2 & 4 & 5 & 1 & 3\\
5 & 1 & 3 & 6 & 2 & 4\\
\hline \end{array}$\quad
$\arraycolsep=3pt\begin{array}{|cccccc|} \hline
5 & 6 & 3 & 4 & 2 & 1\\
2 & 1 & 5 & 6 & 3 & 4\\
1 & 2 & 6 & 5 & 4 & 3\\
3 & 4 & 2 & 1 & 6 & 5\\
4 & 3 & 1 & 2 & 5 & 6\\
6 & 5 & 4 & 3 & 1 & 2\\
\hline \end{array}$\quad
$\arraycolsep=3pt\begin{array}{|cccccc|} \hline
6 & 5 & 4 & 3 & 1 & 2\\
1 & 2 & 6 & 5 & 4 & 3\\
4 & 3 & 2 & 1 & 5 & 6\\
2 & 1 & 5 & 6 & 3 & 4\\
5 & 6 & 3 & 4 & 2 & 1\\
3 & 4 & 1 & 2 & 6 & 5\\
\hline \end{array}$\quad
$\arraycolsep=3pt\begin{array}{|cccccc|} \hline
2 & 1 & 5 & 6 & 3 & 4\\
5 & 6 & 3 & 4 & 2 & 1\\
6 & 5 & 4 & 3 & 1 & 2\\
4 & 3 & 1 & 2 & 5 & 6\\
3 & 4 & 2 & 1 & 6 & 5\\
1 & 2 & 6 & 5 & 4 & 3\\
\hline \end{array}$\quad
$\arraycolsep=3pt\begin{array}{|cccccc|} \hline
1 & 2 & 6 & 5 & 4 & 3\\
6 & 5 & 4 & 3 & 1 & 2\\
3 & 4 & 1 & 2 & 6 & 5\\
5 & 6 & 3 & 4 & 2 & 1\\
2 & 1 & 5 & 6 & 3 & 4\\
4 & 3 & 2 & 1 & 5 & \cellcolor{lightgray}6\\
\hline \end{array}$\quad
    
\end{append}

\section*{Acknowledgement}
Funding: This work was supported by The Australian Research Council, through the Centre of Excellence for Plant Success in Nature and Agriculture (CE200100015) and the second author would like to acknowledge the support of the Australian Government through a Research Training Program (RTP) Scholarship.

\printbibliography

\end{document}